\documentclass[11pt,reqno]{amsart}
\usepackage[margin=1.15in, footskip=.25in]{geometry}
\usepackage[T1]{fontenc}
\usepackage[american]{babel}
\usepackage[europeanresistors, cuteinductors]{circuitikz} 
\usetikzlibrary{decorations.markings}
\tikzstyle arrowstyle=[scale=1]
\tikzstyle directed=[postaction={decorate,decoration={markings,
    mark=at position .65 with {\arrow[arrowstyle]{stealth}}}}]
\tikzstyle reverse directed=[postaction={decorate,decoration={markings,
    mark=at position .65 with {\arrowreversed[arrowstyle]{stealth};}}}]
\usetikzlibrary{shapes.geometric}
\usepackage{amsmath,amsfonts,amsthm,bm}
\usepackage{mathrsfs}
\usepackage{mathtools}
\usepackage{graphicx,subcaption}
\usepackage{hyperref,cite}
\usepackage{caption}
\usepackage{enumerate, enumitem}

\usepackage{float}
\usepackage{array}
\usepackage{courier}
\usepackage[font=small]{caption}
\definecolor{mygreen}{RGB}{25,127,25}

\newtheorem{theorem}{Theorem}[section]
\newtheorem{proposition}[theorem]{Proposition}
\newtheorem{corollary}[theorem]{Corollary}
\newtheorem{lemma}[theorem]{Lemma}
\theoremstyle{definition}
\newtheorem{definition}{Definition}[section]

\newtheorem{assumption}{Assumption}
\newtheorem{remark}{Remark}[section]

\newtheorem*{thm*}{Theorem}
\newtheorem*{lem*}{Lemma}
\newtheorem*{prop*}{Proposition}

\newcommand{\mbbN}{{\mathbb N}}
\newcommand{\mbbZ}{{\mathbb Z}}

\newcommand{\mcB}{\mathcal{B}}
\newcommand{\mcE}{\mathcal{E}}
\newcommand{\mcF}{\mathcal{F}}
\newcommand{\mcD}{\mathcal{D}}

\def\mcC{\mathcal{C}}
\def\mcI{\mathcal{I}}
\def\mcJ{\mathcal{J}}

\def\mcN{\mathcal{N}}

\def\sym{\operatorname{sym}}
\def\asym{\operatorname{sym}^\bot}

\def\Pr{\operatorname{P}\!}

\numberwithin{equation}{section}
\mathtoolsset{showonlyrefs=true}
\title{Heat kernel analysis on diamond fractals}
\author{Patricia Alonso Ruiz}
\address{Department of Mathematics, Texas A{\&}M University, College Station, TX 3368-77843}
\email{paruiz@math.tamu.edu}
\thanks{Research partly supported by the NSF grant DMS~1951577.}

\begin{document}
\begin{abstract}
This paper presents a detailed analysis of the heat kernel on an $(\mbbN\!\times\!\mbbN)$-parameter family of compact metric measure spaces which 
do not satisfy the volume doubling property. 
In particular, uniform bounds of the heat kernel and its Lipschitz continuity, as well as the continuity of the corresponding heat semigroup are studied; a specific example is presented revealing a logarithmic correction. The estimates are further applied to derive several functional inequalities of interest in describing the convergence to equilibrium of the diffusion process.
\end{abstract}

\maketitle
\section{Introduction}
The present paper
investigates the behavior of intrinsic heat diffusion processes in \textit{generalized diamond fractals} through the study of their associated \textit{heat kernel}. 
These fractals 
constitute a parametric family of compact metric measure spaces that arises as a generalization of a hierarchical lattice model appearing in the physics and geometry literature~\cite{Yan88,HK10,AC17}. With a structure reminiscent of the scale irregular fractals treated in~\cite{BH97}, they present some additional non-standard geometric features that make them 
a relevant object of study. 
Specially because 
diamond fractals happen 
to admit a heat kernel with a rather explicit expression~\cite{AR18}, they are most suitable to analyze non-standard model behaviors.
\medskip

Due to their wide range of applications, there is an extensive literature concerning the investigation of heat kernels from different points of view~\cite{JL01,DS01,Gri06,Hof06,Dun12}. 
In this paper, special attention is paid to the rich 
interplay 
between analysis, probability and geometry
that comes to light through 
the study of functional inequalities and estimates related to them, see e.g.~\cite{SC10,BGL14} and references therein.
One of the main reasons to investigate this type of question in the particular setting of generalized diamond fractals is that these spaces, which 
may be described 
via inverse limits of metric measure graphs, 
see Figure~\ref{F:D33}, lack regularity properties such as \textit{volume doubling} or \textit{uniformly bounded degree}, that are often assumed in the literature~\cite{BCG01,BCK05,CK15,CJKS17}. 

\medskip

One of the aims of the paper is thus to set the starting point of a 
larger 
research program
, where 
diamond fractals may be considered as model spaces towards a classification of inverse limit spaces in terms of their heat semigroup properties. On the one hand, this would contribute to the existing research carried out by Cheeger-Kleiner from a more purely geometric point of view in~\cite{CK13,CK15}. On the other hand, some of this analysis may transfer to direct limits of metric measure graphs, so-called \textit{fractal quantum graphs}~\cite{ARKT16}.

\medskip

In order to investigating how the measure-geometric properties of diamond fractals are reflected in the analysis of the diffusion process, the Lipschitz continuity of the heat kernel $p_t$ and the heat semigroup $\{P_t\}_{t\geq 0}$ treated in Section~\ref{S:contHK} and Section~\ref{S:contHS} play a central role in this paper.
Some estimates for the heat kernel in a particular class of diamond fractals were discussed in~\cite[Section 4]{HK10}, however Lipschitz estimates remained unexplored. 
Dealing with
this rather non-standard setting makes
much of the general abstract theory not directly applicable, and
being able to work with explicit expressions will become crucial to approach its analysis. As an example, in the particular case of a (regular) diamond fractal with parameters $n$ and $j$, Corollary~\ref{C:HKsup_regular} provides the estimate

\begin{equation*}
\frac{1}{\sqrt{4\pi }}t^{-1/2}\leq \|p_t\|\leq \frac{1}{2\pi}+\frac{1}{\sqrt{4\pi }}t^{-1/2}+C_{n,j}\,t^{-\frac{1}{2}\big(1+\frac{\log n}{\log j}\big)}
\end{equation*}
with a constant $C_{n,j}$ that can be explicitly bounded.
Continuity estimates of the heat semigroup are deeply connected to the geometry of the underlying space. This connection is displayed for instance in so-called \textit{Bakry-\'Emery type curvature conditions}. 
In the classical setting of a complete and connected Riemmanian manifold, 
such a condition can be expressed as an inequality involving the gradient of the semigroup and it is known to be equivalent to a bound of the Ricci curvature of 
the space~\cite{Bak97,Led00,vRS05}. In recent years, a significant amount of research has been carried out to characterize curvature bounds in the context of metric measure spaces and Dirichlet spaces 
by means of weak versions of the original Bakry-\'Emery condition, see e.g.~\cite{LV07,Stu06,Stu06b,AGS15,AGS17}. 

\medskip

The lack of the volume doubling property leaves generalized diamond fractals out of general frameworks such as the Alexandrov spaces with curvature bounded from below treated in~\cite{GKO13}. 
As a first approach to investigate this type of connection with curvature in the present setting, Section~\ref{S:contHS} deals with the regularity of the heat semigroup $\{P_t\}_{t\geq 0}$ and its relation to the so-called \textit{weak Barky-\'Emery nonnegative curvature condition} recently introduced in the framework of Dirichlet spaces with heat kernel bounds~\cite{ABCRST3}. 
The most concrete computable case presented in this paper,
see~Theorem~\ref{T:BE_reg_log}, 
reveals a logarithmic correction term
\begin{equation*}
|P_tf(x)-P_t f(y)|\leq C\frac{|\log t|}{\sqrt{t}}d(x,y)\|f\|_\infty,\qquad 0<t<1,
\end{equation*}
which reflects 
the inhomogeneous nature of diamond fractals that allows the measure to be very different at different points. This type of phenomenon is observed in diffusion processes with multifractal structures, see e.g.~\cite{BK01}.

\medskip

To further illustrate the applications of the continuity estimates obtained, the last section briefly discusses
several functional inequalities related to the diffusion process that sometimes involve the \textit{Dirichlet form} associated with the process. 
Section~\ref{S:igDF} is devoted to the investigation of relevant potential analytic properties of the diffusion 
process such as the characterization of the Dirichlet form on a diamond fractal as a suitable (Mosco) limit of Dirichlet forms on its approximating metric measure graphs; c.f.~Theorem~\ref{T:DF_props}. 
The results mainly rely on the observation that ``liftings'' of functions in the approximating metric measure graphs provide a very natural core of functions, see Theorem~\ref{T:core_ig}; this property holds in the more general setting of inverse limits and will be the subject of a forthcoming paper. 

\medskip

Apart from the inequalities outlined in Section~\ref{S:ineqs}, it is noteworthy to point out the fact that generalized diamond fractals do not satisfy the \textit{elliptic Harnack inequality}. This was proved in~\cite{HK10} for a particular (self-similar) class of diamond fractals and in general it follows directly from the recent result~\cite[Theorem 3.11]{BM18} since generalized diamonds are not \textit{metric doubling}, see~\cite{BM18} and references therein.

%
%
%
\medskip

The paper is organized as follows: Section~\ref{S:GDF} starts with a brief review of the inverse limit construction of generalized diamond fractals that was carried out in~\cite{AR18}. Some basic and practical metric properties are discussed and the working explicit expression of the heat kernel is given in Theorem~\ref{T:HK_Finfty}. Section~\ref{S:igDF} investigates potential theoretical aspects of the diffusion process and its relation with the inverse limit structure in terms of the infinitesimal generator and the Dirichlet form. In particular, Theorem~\ref{T:core_ig} identifies a core of functions that is widely used in the subsequent analysis. The main results of the paper are concentrated in Section~\ref{S:contHK} and Section~\ref{S:contHS}. Theorem~\ref{T:supHK} provides a general uniform estimate of the heat kernel, whereas Theorem~\ref{T:LipFinfty} and Theorem~\ref{T:BE_Finfty} deal with the Lipschitz continuity in space of the heat kernel and of the heat semigroup, respectively. To gain a better understanding of their time-dependence, the results are applied to a particular class of diamond fractals for which computations become more tractable, c.f.\ Theorem~\ref{T:BE_reg_log}. Finally, Section~\ref{S:ineqs} outlines further applications of the estimates to study the semigroup and the diffusion process, namely logarithmic Sobolev inequality, ultracontractivity and Poincar\'e inequalities.

\section{Generalized diamond fractals}\label{S:GDF}
This section sets up notation and briefly reviews the construction of generalized diamond fractals as inverse limits of metric measure graphs presented in~\cite{AR18}. We point out some metric observations that will become useful when dealing with estimates in later sections, and summarize the key results concerning the natural diffusion process associated with these spaces. In particular, Lemma~\ref{L:Intertwin} and Theorem~\ref{T:HK_Finfty} restate crucial facts about the heat semigroup and the heat kernel that are essential to the analysis carried out subsequently.

\subsection{Inverse limit construction}
A diamond fractal, see Figure~\ref{F:D33}, arises from a sequence of metric measure graphs and is characterized by two parameter sequences $\mcJ=\{j_i\}_{i\geq 0}$ and $\mcN=\{n_i\}_{i\geq 0}$ that describe its construction. Each sequence indicates, respectively, the number of new vertices added from one graph to its next generation, and the number of additional edges given to each vertex. 

\begin{figure}[H]
\centering
\renewcommand{\arraystretch}{0.1}
\begin{tabular}{p{0.15\textwidth}p{0.2\textwidth}p{0.2\textwidth}}
$F_0$ & $F_1$ & $F_2$\\
\begin{tikzpicture}
\foreach \a in {0,60,120,180,240,300}{
\draw ($(\a:1)$) to[out=90+\a,in=330+\a] ($(60+\a:1)$);
\fill[color= blue] ($(\a+60:1)$) circle (1pt);
}
\end{tikzpicture}
&
\begin{tikzpicture}
\foreach \a in {0,60,120,180,240,300}{
\draw ($(\a:1)$) to[out=90-55+\a,in=15+\a] ($(60+\a:1)$);
\draw ($(\a:1)$) to[out=90+\a,in=330+\a] ($(60+\a:1)$);
\draw ($(\a:1)$) to[out=90+65+\a,in=-90+5+\a] ($(60+\a:1)$);
\fill[color= blue] ($(\a+60:1)$) circle (1pt);
\fill[color= blue] ($(30+\a:1)$) circle (1pt);
\fill[color= blue] ($(30+\a:1.15)$) circle (1pt);
\fill[color= blue] ($(30+\a:0.725)$) circle (1pt);
}
\end{tikzpicture}
&
\begin{tikzpicture}
\foreach \a in {0,60,120,180,240,300}{
\draw ($(\a:1)$) to[out=30+\a,in=340+\a] ($(30+\a:1.15)$);
\draw ($(\a:1)$) to[out=50+\a,in=320+\a] ($(30+\a:1.15)$);
\draw ($(\a:1)$) to[out=70+\a,in=300+\a] ($(30+\a:1.15)$);
\draw ($(\a:1)$) to[out=\a-30,in=\a-340] ($(\a-30:1.15)$);
\draw ($(\a:1)$) to[out=\a-50,in=\a-320] ($(\a-30:1.15)$);
\draw ($(\a:1)$) to[out=\a-70,in=\a-300] ($(\a-30:1.15)$);
}
\foreach \a in {0,60,120,180,240,300}{
\draw ($(\a:1)$) to[out=70+\a,in=320+\a] ($(30+\a:1)$);
\draw ($(\a:1)$) to[out=90+\a,in=300+\a] ($(30+\a:1)$);
\draw ($(\a:1)$) to[out=110+\a,in=280+\a] ($(30+\a:1)$);
\draw ($(\a:1)$) to[out=\a-70,in=\a-320] ($(\a-30:1)$);
\draw ($(\a:1)$) to[out=\a-90,in=\a-300] ($(\a-30:1)$);
\draw ($(\a:1)$) to[out=\a-110,in=\a-280] ($(\a-30:1)$);
}
\foreach \a in {0,60,120,180,240,300}{
\draw ($(\a:1)$) to[out=110+\a,in=320+\a] ($(30+\a:.725)$);
\draw ($(\a:1)$) to[out=130+\a,in=300+\a] ($(30+\a:.725)$);
\draw ($(\a:1)$) to[out=150+\a,in=280+\a] ($(30+\a:.725)$);
\draw ($(\a:1)$) to[out=\a-110,in=\a-320] ($(\a-30:.725)$);
\draw ($(\a:1)$) to[out=\a-130,in=\a-300] ($(\a-30:.725)$);
\draw ($(\a:1)$) to[out=\a-150,in=\a-280] ($(\a-30:.725)$);
\fill[color= blue] ($(\a+60:1)$) circle (1pt);
\fill[color= blue] ($(30+\a:1)$) circle (1pt);
\fill[color= blue] ($(30+\a:1.15)$) circle (1pt);
\fill[color= blue] ($(30+\a:0.725)$) circle (1pt);
}
\end{tikzpicture}
\end{tabular}
\caption{\small{Approximations of a diamond fractal with parameters $j_1,n_1=3,j_2=2,n_2=3$.}}
\label{F:D33}
\end{figure}
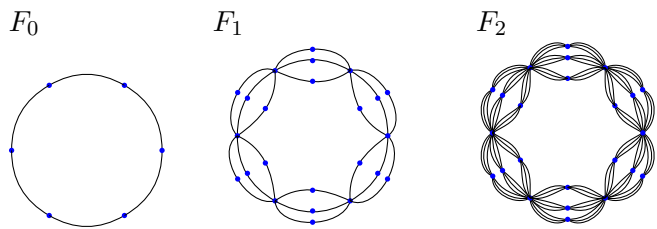

\begin{definition}\label{D:prod_seq}
Let $\mcJ=\{j_\ell\}_{\ell\geq 0}$, $\mcN=\{n_\ell\}_{\ell\geq 0}$ be sequences with $j_0=1=n_0$ and $j_\ell,n_\ell\geq 2$ for all $\ell\geq 1$. Set $J_0=N_0=1$ and define for any $0\leq k\leq i$ the products
\begin{equation*}
J_{k,i}:=\prod_{\ell=k}^ij_\ell,\qquad\qquad N_{k,i}:=\prod_{\ell=k}^in_\ell.
\end{equation*}
In particular, we write $J_i:=J_{0,i}$ and $N_i:=N_{0,i}$.
\end{definition}

The inverse system associated with a diamond fractal is built upon a sequence of metric measure spaces $(F_i,d_i,\mu_i)$ that can be defined inductively in the following manner.
\begin{definition}\label{D:Fi}
Let $F_0$ denote the unit circle and set $\vartheta_0:=\{0,\pi\}$, $B_0:=\vartheta_0$. For each $i\geq 1$, define $\vartheta_i:=\big\{\frac{\pi k}{J_i}~|~0< k<2J_i,~k\hspace*{-.5em}\mod j_i\not\equiv 0\big\}$, set $B_1=B_0\cup\vartheta_1$ and
\begin{equation*}
B_i:=B_{i-1}\cup(\vartheta_i\times[n_1]\times\ldots\times[n_{i-1}])\qquad i\geq 2,
\end{equation*}
with $[n_k]=\{1,\ldots,n_k\}$. For each $i\geq 1$, we define the quotient space
\[
F_i:=F_{i-1}\times[n_i]~/~{\stackrel{i}{\sim}},
\]
where $xw\stackrel{i}{\sim} x'w'$ if and only if $x,x'\in B_i$.
\end{definition}
The set $B_i$ contains the identification (branching, junction) points that yield $F_i$ and satisfies $B_i\subseteq F_{i-1}$, see marked dots in Figure~\ref{F:D33}. As a metric measure graph, each $F_i$ can be regarded as the union of branches ($i$-cells) isomorphic to intervals of length $\pi/J_i$ that suitably connect the vertices in $B_{i-1}$. The measure $\mu_i$ that is naturally induced on each $F_i$ is obtained by redistributing the mass of each branch in the previous level uniformly between its ``successors''. The corresponding (geodesic) distance $d_i$ on $F_i$ coincides with the Euclidean metric on each branch.

\medskip

Following the construction scheme from Definition~\ref{D:Fi} one can produce a family of measurable mappings $\phi_{ik}\colon F_i\to F_k$, $0\leq k\leq i$, in such a way that the sequence $\{(F_i,\mu_i,\{\phi_{ik}\}_{k\leq i})\}_{i\geq 0}$ defines an inverse (projective) system of measure spaces. We refer to~\cite[Section 2]{AR18} for a precise definition of inverse systems and of the mappings $\phi_{ik}$, summarized in Figure~\ref{F:projlim}. 

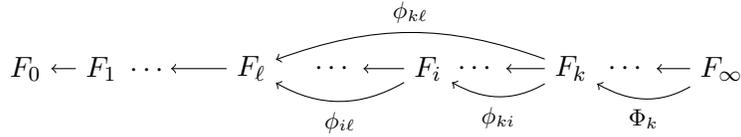
\begin{figure}[H]
\begin{tikzpicture}
\node (F0) at ($(0:0)$) {$F_0$};
\node (F1) at ($(0:1.35)$) {$F_1\;\cdots$};
\node (dots) at ($(0:1.8)$) {};
\node (Fl) at ($(0:3)$) {$F_\ell$};
\node (Fil) at ($(0:4.1)$) {$\cdots$};
\node (Fi) at ($(0:5.35)$) {$F_i$};
\node (dotss) at ($(0:6)$) {$\cdots$};
\node (Fk) at ($(0:7.25)$) {$F_k$};
\node (dotsss) at ($(0:8)$) {$\cdots$};
\node (Finf) at ($(0:9.25)$) {$F_\infty$};

\draw[->] (F1) -- (F0);
\draw[<-] (dots) -- (Fl);
\draw[->] (Finf) -- (dotsss);
\draw[->] (Fi) -- (Fil);
\draw[->] (Fk) -- (dotss);

\draw[<-] (Fl) to[out=330,in=210] node[below] {\footnotesize{$\phi_{i\ell}$}} (Fi);
\draw[<-] (Fl)  to[out=20,in=160] node[above] {\footnotesize{$\phi_{k\ell}$}} (Fk);
\draw[<-] (Fi)  to[out=330,in=210] node[below] {\footnotesize{$\phi_{ki}$}} (Fk);
\draw[<-] (Fk)  to[out=330,in=210] node[below] {\footnotesize{$\Phi_{k}$}} (Finf);
\end{tikzpicture}
\vspace*{-1em}
\caption{Projective system structure.}
\label{F:projlim}
\end{figure}

A generalized diamond fractal of parameters $\mcJ$ and $\mcN$ will arise as the inverse (projective) limit of the above-mentioned inverse system. In this way, the limit space $(F_\infty,\mu_\infty)$ is equipped with a family of measurable ``projection mappings'', $\Phi_i\colon F_\infty\to F_i$, that will play a major role in the construction of the associated diffusion process.
To fully realize a diamond fractal as a metric measure space, we briefly discuss in the next paragraph the metric that naturally comes along with the inverse limit construction.
\subsection{Metric remarks}
By definition, the graphs $F_i$ are equipped with the geodesic metric $d_i$ induced by the Euclidean on each edge. The following observation describes how metrics in different levels are related by means of the mappings $\phi_{ik}$ described in the previous section, and it justifies the definition of the metric on the limit space $F_\infty$. For ease of the notation, we write for each $i\geq 1$, $\phi_i:=\phi_{i(i-1)}\colon F_i\to F_{i-1}$.
\begin{lemma}\label{L:dist}
For any $i\geq 1$ and $x,y\in F_i$ it holds that
\begin{equation}\label{E:dist}
d_{i-1}(\phi_i(x),\phi_i(y))\leq d_i(x,y)\leq d_{i-1}(\phi_i(x),\phi_i(y))+2\pi/J_i.
\end{equation}
Moreover, $\displaystyle d_k(\phi_{ik}(x),\phi_{ik}(y))\leq d_i(x,y)$ for any $0\leq k\leq i$.
\end{lemma}
\begin{proof}
Notice that the length of a branch in level $i$ is $\pi/J_i$. Inequality~\eqref{E:dist} thus follows by construction. Applying the left hand side of~\eqref{E:dist} repeatedly and using the fact that $\phi_{ik}=\phi_{k+1}{\circ}\cdots{\circ}\phi_i$ we obtain
\begin{align*}
d_k(\phi_{ik}(x),\phi_{ik}(y))&= d_k(\phi_{k+1}(\phi_{i(k+2)}(x)),\phi_{k+1}(\phi_{i(k+2)}(y)))\leq d_{k+1}(\phi_{i(k+2)}(x),\phi_{i(k+2)}(y))\\
&\leq\ldots\leq d_{i-1}(\phi_i(x),\phi_i(y))\leq d_i(x,y).
\end{align*}
\end{proof}
The metrics $d_i$ also satisfy a chain property that will be particularly useful to compute estimates in later sections.
\begin{lemma}\label{L:chain_cond}
For any $x,y\in F_i$ there exist $z_1,\ldots, z_{m_{xy}}\in B_i$ with $1\leq m_{xy}\leq J_i$  and such that
\begin{equation*}
d_i(x,y)=d_i(x,z_1)+\sum_{\ell=1}^{m_{xy}-1}d_i(z_\ell,z_{\ell+1})+d_i(z_{m_{xy}},y).
\end{equation*}
\end{lemma}
As a direct consequence of Lemma~\ref{L:dist}, for any $x,y\in F_\infty$ the sequence $\{d_i(\Phi_i(x),\Phi_i(y))\}_{i\geq 0}$ converges uniformly and we may thus define
\begin{equation}\label{E:def_dist}
d_\infty (x,y)=\lim_{i\to\infty}d_i(\Phi_i(x),\Phi_i(y))
\end{equation}
to be the natural metric that carries the inverse limit structure of the generalized diamond $F_\infty$.

\begin{definition}
Let $\mcJ=\{j_\ell\}_{\ell\geq 0}$ and $\mcN=\{n_\ell\}_{\ell\geq 0}$ be sequences with $j_0=n_0=1$ and $j_\ell,n_\ell\geq 2$. The generalized diamond fractal of parameters $\mcJ$ and $\mcN$ is defined to be the inverse limit of the system $\{(F_i,d_i,\mu_i,\{\phi_{ik}\}_{k\leq i})\}_{i\geq 0}$. If $j_\ell=j$ and $n_\ell=n$ for some $j,n\geq 2$ and all $\ell\geq 1$, we say that the generalized diamond fractal is regular.
\end{definition}

We finish this paragraph by observing that~\eqref{E:dist} together with the fact that the mappings $\Phi_i$ are surjective readily implies the convergence in the \textit{pointed measured Gromov-Hausdorff sense} of the inverse system; c.f.~\cite[Proposition 2.17]{CK15}.
\begin{proposition}
A generalized diamond fractal $(F_\infty,d_\infty,\mu_\infty)$ is the inverse limit and the limit in the pointed measured Gromov-Hausdorff sense of $\{(F_i,d_i,\mu_i)\}_{i\geq 0}$.
\end{proposition}

\subsection{Diffusion process and heat kernel}
Barlow and Evans proved in~\cite{BE04} the existence of a diffusion process associated with an inverse limit. Following a procedure from~\cite{BPY89}, it was proved in~\cite{AR18} that for a very large class of parameter sequences, it is actually possible to provide a rather explicit expression of the heat kernel associated with the natural diffusion on a diamond fractal; see~Theorem~\ref{T:HK_Finfty}. This paragraph summarizes the results obtained in~\cite{AR18}; their application in the analysis of the process and its heat kernel are the main object of study in the present paper. 

\medskip

In order to provide later on estimates that are expressible in a ``classical'' form, 
the paper parameter sequences $\mcN=\{n_i\}_{i\geq 0}$ and $\mcJ=\{j_i\}_{i\geq 0}$ under consideration will satisfy the following weak condition.

\begin{assumption}\label{A:NJ_ass2}
For any fixed $k\geq 0$, $\displaystyle \lim_{\ell\to\infty}n_\ell j_\ell e^{-J_{k+1,\ell-1}^2}=0$ 
and in particular the series
$
\sum\limits_{\ell=k}^\infty N_{k,\ell}J_{k,\ell}e^{-J_{k+1,\ell}^2}
$
converges and can be bounded uniformly on $k\geq 0$.
\end{assumption}

The latter assumption is readily satisfied for regular sequences, where $j_i=j$ and $n_i=n$ for some fixed $j,n\geq 2$, see e.g. Corollary~\ref{C:HKsup_regular}. The somewhat weaker condition
\begin{equation}\label{E:NJ_ass}
\lim_{i\to\infty}N_ie^{-J_i^2t}<\infty
\end{equation}
for $0<t<t_*<1$ 
already provides the existence of a jointly continuous heat kernel~\cite[Remark 3.1]{AR18} and general estimates in terms of series, see Remark~\ref{R:series_bound}. However, little information about the convergence of those series can be obtained without further assumptions. 

\medskip

Several results presented in the subsequent sections will involve the $L^2$-semigroup associated with the diffusion processes on $F_i$, $i=1,2,\ldots,\infty$, which we denote by $\{P_t^{F_i}\}_{t\geq 0}$. The next lemma states an intertwining property that relates these semigroups with the mappings 
\begin{align}
\Phi_i^*\colon &L^2(F_i,\mu_i)\;\longrightarrow \;L^2(F_\infty,\mu_\infty)\label{E:Phi_i_ast}\\
&\qquad f\quad\quad\longmapsto\quad f\circ \Phi_i,\nonumber
\end{align}
and will be applied crucially in several occasions. 

\begin{lemma}\cite[Lemma 3, Corollary 4]{AR18}\label{L:Intertwin}
The family of operators $\{P^{F_\infty}_t\}_{t\geq 0}$ is a strongly continuous Markov semigroup on $L^2(F_\infty,\mu_\infty)$ that satisfies the strong Feller property. Moreover, for any $i\geq 0$,
\begin{equation}\label{E:Intertwin}
P_{t}^{F_\infty}\Phi_{i}^*f=\Phi_i^*P_t^{F_i}f
\end{equation}
holds for any $f\in L^2(F_i,\mu_i)$.
\end{lemma}

Each of these semigroups admits a jointly continuous heat kernel 
that turns out to be expressible in terms of the heat kernel on the circle and on intervals $[0,\pi/J_i]$ with Dirichlet boundary conditions, denoted by $p^{F_0}_t$, respectively $p_t^{[0,\pi/J_{i}]_D}$. Since these heat kernels appear repeatedly in later computations, some standard estimates and facts about them are recorded in Appendix~\ref{S:Apx_HKcircle}. 
\begin{theorem}\label{T:HK_Finfty}
The heat kernel associated with the semigroup $\{P_t^{F_i}\}_{t\geq 0}$ 
is given by
\begin{equation}\label{E:HK_general_form}
p_t^{F_\infty}(x,y)=p_t^{F_0}\big(\Phi_{0}(x),\Phi_{0}(y)\big)
+\sum_{\ell=1}^{i_{xy}}\delta_{xy}(n_\ell)N_{\ell-1}p_t^{[0,\pi/J_{\ell}]_D}(\Phi_0(x),\Phi_0(y))
\end{equation}
for any $x,y\in F_\infty$, where $i_{xy}:=\max_{i\geq 0}\{\Phi_{i}(x),\Phi_{i}(y)\text{ belong to the same bundle}\}$ and 
\begin{equation*}
\delta_{xy}(n)=\begin{cases}n-1&\text{if }\Phi_{i_{xy}}(x),\Phi_{i_{xy}}(y)\text{ same branch,}\\
 -1&\text{if }\Phi_{i_{xy}}(x),\Phi_{i_{xy}}(y)\text{ same bundle, different branch.}
 \end{cases}
\end{equation*}
\end{theorem}

\begin{proof}
The recursive formula provided in~\cite[Theorem 2]{AR18} can be rewritten as
\begin{equation*}
p_t^{F_i}(x,y)=p_t^{F_0}\big(\phi_{i0}(x),\phi_{i 0}(y)\big)+\sum_{\ell=1}^{i_{xy}}\delta_{xy}(n_\ell)N_{\ell-1}J_{\ell}\big(p_{J_\ell^2t}^{F_0}\big(\phi_{i0}(x),\phi_{i0}(y)\big){-}p_{J_\ell^2t}^{F_0}\big(\phi_{i0}(x),-\phi_{i0}(y)\big)\big),
\end{equation*}
where $i_{xy}:=\max_{0\leq k\leq i}\{\phi_{ik}(x),\phi_{ik}(y)\text{ belong to the same bundle}\}$ and $\delta^{(i)}_{xy}(n)$ is as $\delta_{xy}(n)$ with $\phi_{i,i_{xy}}$ instead of $\Phi_{i_{xy}}$.
Using the relation~\eqref{E:Dir_HKvsF0} to rewrite $p_{J_\ell^2t}^{F_0}$ in terms of the heat kernel on an interval and  
noting that by definition of inverse limit $\phi_{ik}(\Phi_i(x))=\Phi_k(x)$ for any $k\geq 0$,~\cite[Theorem 3]{AR18} gives~\eqref{E:HK_general_form}.
\end{proof}
At this point the reader may already wonder, whether these results could readily be extended to straightforwardly generalized constructions, such as considering different branch lengths and numbers within each approximation level. These could certainly be handled with the same techniques used in this paper. However, at the moment, 
setting up a formula including all possibilities 
seems to 
result in a fairly long and possibly painfully readable jungle of notation without providing greater insight. 

\section{Infinitesimal generator and Dirichlet form}\label{S:igDF}
As a strongly continuous Markov semigroup on $L^2(F_i,\mu_i)$, each $\{P^{F_i}_t\}_{t\geq 0}$, $i=1,\ldots,\infty$ has an associated infinitesimal generator and a Dirichlet form, which we denote by $L_{F_i}$ and $(\mcE^{F_i},\mcF^{F_i})$ respectively. In particular the semigroup and the Dirichlet form will appear in the functional inequalities discussed in the subsequent sections. The present section aims to record several properties of interest and to serve as motivation and the starting point for a forthcoming paper, where these questions will be analyzed in the more general setting of \textit{inverse limit spaces}.

\subsection{Liftings and projections}
The mappings that provided the intertwining relation between the semigroups $\{P^{F_\infty}_t\}_{t\geq 0}$ and $\{P^{F_i}_t\}_{t\geq 0}$ from Lemma~\eqref{L:Intertwin} will play a major role in the subsequent discussion. 
\begin{proposition}\label{P:Ana_props} 
Let $i\geq 0$ and $\Phi_i^*\colon L^2(F_i,\mu_i)\to L^2(F_\infty,\mu_\infty)$ be defined as in~\eqref{E:Phi_i_ast}.
\begin{enumerate}[wide=0em,itemsep=.5em,label={\rm (\roman*)}]
\item For each $i\geq 0$, $\Phi_i^*$ is an isometry.
\item The space $\mcC_0:=\bigcup_{i\geq 0}\Phi_i^*C(F_i)$ is dense in $L^2(F_\infty,\mu_\infty)$.
\end{enumerate}
\end{proposition}
\begin{proof}
These statements follow by the definition of $\Phi_i^*$, see e.g.~\cite[Proposition 2]{AR18}.
\end{proof}
\begin{definition}\label{D:Proj_i}
For each $i\geq 0$, let $\Pi_i\colon L^2(F_\infty,\mu_\infty)\to L^2(F_i,\mu_i)$ be the left inverse of $\Phi_i^*$, i.e.
\begin{enumerate}[leftmargin=.5em,label={\rm(\alph*)}, wide=0em, itemsep=.5em]
\item\label{D:Left_inv_a} $\displaystyle \langle \Pi_i f,f_i\rangle_{L^2(F_i,\mu_i)}=\langle f,\Phi_i^* f_i\rangle_{L^2(F_\infty,\mu_\infty)}$ for any $f\in L^2(F_\infty,\mu_\infty),\;f_i\in L^2(F_i,\mu_i)$;
\item\label{D:Left_inv_b} $\displaystyle \Pi_i\Phi_i^*f_i=f_i$ for any $f_i\in L^2(F_i,\mu_i)$.
\end{enumerate} 
\end{definition}

While $\Phi_i^*$ may be understood as a ``lifting'', its left inverse $\Pi_i$ is in fact a projection mapping. In the next lemma, we prove some of its most relevant properties.
\begin{proposition}\label{P:Pi_i_props}
For any $i\geq 0$ and $f\in L^2(F_\infty,\mu_\infty)$, 
\begin{enumerate}[wide=0em,itemsep=.5em,label={\rm (\roman*)}]
\item $\displaystyle \|\Pi_i f\|_{L^2(F_i,\mu_i)}\leq\|f\|_{L^2(F_\infty,\mu_\infty)}$;
\item\label{E:L2sps_conv} $\displaystyle \|f\|_{L^2(F_\infty,\mu_\infty)}=\lim_{i\to\infty}\|\Pi_i f\|^2_{L^2(F_i,\mu_i)}$.
\end{enumerate}
\end{proposition} 

\begin{proof}
\begin{enumerate}[leftmargin=.5em,label=(\roman*), wide=0em, itemsep=.5em]
\item Since $\Pi_i$ is the left inverse of $\Phi_i^*$, applying Cauchy-Schwartz we have 
\begin{align*}
\|\Pi_i f\|_{L^2(F_i,\mu_i)}^2&=\langle \Pi_i f,\Pi_i f\rangle_{L^2(F_i,\mu_i)}=\langle f,\Phi_i^*\Pi_if\rangle_{L^2(F_\infty,\mu_\infty)}\\
&\leq \|f\|_{L^2(F_\infty,\mu_\infty)}\|\Phi_i^*\Pi_i f\|_{L^2(F_\infty,\mu_\infty)}=\|f\|_{L^2(F_\infty,\mu_\infty)}\|\Pi_i f\|_{L^2(F_\infty,\mu_\infty)}.
\end{align*}
\item Again by Cauchy-Schwartz,
\begin{align*}
\big|\,\|f\|^2_{L^2(F_\infty,\mu_\infty)}-\|\Pi_i f\|^2_{L^2(F_i,\mu_i)}\big|
&=\big|\langle f,f-\Phi_i^*\Pi_i f\rangle_{L^2(F_\infty,\mu_\infty)}\big|\\
&\leq\|f\|_{L^2(F_\infty,\mu_\infty)}\|f-\Phi_i^*\Pi_i f\|_{L^2(F_\infty,\mu_\infty)}.
\end{align*}
Recall from Proposition~\ref{P:Ana_props} that $\mcC_0$ is dense in $L^2(F_\infty,\mu_\infty)$. Let thus $\{f_i\}_{i\geq 0}$ be a sequence such that $f_i\in C(F_i)$ and $\Phi_i^* f_i\xrightarrow{i\to\infty}f$ in $L^2$. Then, using the triangle inequality, the isometry property of $\Phi_i^*$, the definition of left inverse and part (i) of the present proof we get
\begin{align}
\|f-\Phi_i^*\Pi_i f\|_{L^2(F_\infty,\mu_\infty)}&\leq \|f-\Phi_i^* f_i\|_{L^2(F_\infty,\mu_\infty)}+\|\Phi_i^*(f_i-\Pi_i f)\|_{L^2(F_\infty,\mu_\infty)}\nonumber\\
&\leq\|f-\Phi_i^* f_i\|_{L^2(F_\infty,\mu_\infty)}+\|\Phi_i^*f_i-f\|_{L^2(F_\infty,\mu_\infty)}\label{E:Pi_i_props_e1}
\end{align}
which vanishes as $i\to\infty$.
\end{enumerate}
\end{proof}
In particular,~\ref{E:L2sps_conv} in the latter proposition implies the convergence of the $L^2$-spaces $L^2(F_i,\mu_i)$ as introduced in~\cite[Definition 2.1]{Kol06}, see also~\cite{KS03}: a sequence of Hilbert spaces $\{H_i\}_{i\geq 0}$ is said to converge to another Hilbert space $H$ if there exists a dense subspace $C\subset H$ and a sequence of operators $A_i\colon C\to H_i$ such that 
\begin{equation}
 \|f\|_{H}=\lim_{i\to\infty}\|A_i f\|^2_{H_i}
\end{equation}
for every $f\in C$. In the present setting, convergence follows by choosing $H=L^2(F_\infty,\mu_\infty)$, $H_i=L^2(F_i,\mu_i)$, $C=\mcC_0$ and $A_i=\Pi_i$.

\begin{corollary}\label{C:L2sps_conv}
The sequence of spaces $\{L^2(F_i,\mu_i)\}_{i\geq 0}$ converges to $L^2(F_\infty,\mu_\infty)$. 
\end{corollary}
We finish this paragraph by analyzing the combined action of the lifting $\Phi_i^*$, the semigroup $\{P^{F_i}_t\}_{t\geq 0}$ and the projection $\Pi_i$ through the operator $\Phi_i^*P^{F_i}_t\Pi_i\colon L^2(F_\infty,\mu_\infty)\to L^2(F_\infty,\mu_\infty)$.
This will be useful later, in particular to derive the Mosco convergence of the associated Dirichlet forms.

\begin{lemma}\label{L:Sg_strong_conv}
For any $t\geq 0$, the sequence of bounded operators $\{\Phi_i^*P^{F_i}_t\Pi_i\}_{i\geq 0}$ converges strongly  in $L^2(F_\infty,\mu_\infty)$ to $P^{F_\infty}_t$. In particular, the convergence is uniform in any finite time interval.
\end{lemma}
\begin{proof}
By virtue of Lemma~\ref{L:Intertwin} and the contraction property of the semigroup $P^{F_\infty}_t$, for any $f\in L^2(F_\infty,\mu_\infty)$ we have
\begin{equation*}
\|\Phi_i^*P^{F_i}_t\Pi_if-P^{F_\infty}_tf\|_{L^2(F_\infty,\mu_\infty)}=\|P^{F_\infty}_t\Phi_i^*\Pi_if-P^{F_\infty}_tf\|_{L^2(F_\infty,\mu_\infty)}\leq \|\Phi_i^*\Pi_if-f\|_{L^2(F_\infty,\mu_\infty)}.
\end{equation*}
The latter tends to zero as $i\to\infty$ in view of~\eqref{E:Pi_i_props_e1}. The convergence is independent of $t\geq 0$ and hence uniform on any finite interval.
\end{proof}
\subsection{Infinitesimal generator}
Since the finite approximations $F_i$ are metric graphs, for finite indexes $i\geq 0$ the operator $L_{F_i}$ with domain $\mcD_{F_i}$ corresponds with the standard Laplacian studied in quantum graphs/cable systems; see e.g.~\cite{BB04,BK13}. In this paragraph we focus on the properties of the generator $L_{F_\infty}$ and its domain $\mcD_{F_\infty}$ that can be obtained from the results in the previous paragraph. 

\medskip

Finding an explicit characterization of the domain is usually a delicate and difficult question. 
Luckily, for many purposes it will be enough to have a suitable dense set of functions at hand. The next theorem 
provides a natural \textit{core} of functions for $L_{F_\infty}$, i.e.\ a subspace $\mcD_0$ of the domain $\mcD_{F_\infty}$ with the property that the closure of the restriction $L_{F_\infty}|_{\mcD_0}$ coincides with $L_{F_\infty}$.

\begin{theorem}\label{T:core_ig}
For each $i\geq 0$, let $\mcD_{F_i}$ denote the domain of infinitesimal generator $L_{F_i}$. 
The space $\mcD_0:=\bigcup_{i\geq 1}\Phi_i^*\mcD_{F_i}$ 
is a core for the infinitesimal generator $(L_{F_\infty},\mcD_{F_\infty})$.
\end{theorem}
\begin{proof}
Let $f\in \mcD_0$. Then, $f=\Phi^*_ih$ for some $h\in \mcD_{F_i}$ and $i\geq 0$. By Lemma~\ref{L:Intertwin},
\begin{equation}
P^{F_\infty}_tf=P^{F_\infty}_t\Phi^*_ih=\Phi_i^*P_t^{F_i}h\in \Phi_i^*\mcD_{F_i}\subseteq\mcD_0.
\end{equation}
We have thus proved that $P^{F_\infty}_t\colon\mcD_0\to\mcD_0$. Furthermore, $C^\infty(F_i)$ is dense in $C(F_i)$ as well as in $\mcD_{F_i}$ for each $i\geq 1$, 
so 
Proposition~\ref{P:Ana_props} implies that 
$\bigcup_{i\geq 1}\Phi_i^*C^\infty(F_i)$, hence
$\mcD_0$, is dense in $L^2(F_\infty,\mu_\infty)$. By virtue of~\cite[Section 1, Proposition 3.3]{EK86} $\mcD_0$ is a core for the infinitesimal generator of $P^{F_\infty}_t$.
\end{proof}
In other words, any function $f\in\mcD_{F_\infty}$ can be approximated by a sequence $\{f_i\}_{i\geq 0}\subseteq\mcD_0$ so that $\|f-f_i\|_{L^2(F_\infty,\mu_\infty)}\to 0$ and $\|L_{F_\infty} f-L_{F_\infty}f_i\|_{L^2(F_\infty,\mu_\infty)}\to 0$ as $i\to \infty$.

\medskip

We can now use the previous paragraph to describe the relation between the lifting and projection maps and the inifinitesimal generator. 

\begin{corollary}
For each $f\in\mcD_0$, there exists $\{f_i\}_{i\geq 0}$ with $f_i\in \mcD_{F_i}$ such that 
\begin{equation*}
\Phi^*_if_i\xrightarrow{i\to\infty}f\qquad\text{and}\qquad\Phi^*_iL_{F_i}f_i\xrightarrow{i\to\infty}L_{F_\infty}f
\end{equation*}
hold in $L^2(F_\infty,\mu_\infty)$.
\end{corollary}
\begin{proof}
Apply Lemma~\ref{L:Sg_strong_conv} to~\cite[Theorem 2.5]{Kim06}.
\end{proof}

\begin{remark}
In view of~\cite[Theorem 2.5]{Kim06}, the latter result, or equivalently Lemma~\ref{L:Sg_strong_conv}, provides an analogous statement for the resolvent, that already appeared in~\cite[Theorem 4.3]{BE04}.  
Further spectral properties of the infinitesimal generator $L_{F_\infty}$ can be derived from the more abstract setting discussed in~\cite[Section 5]{ST12}.
\end{remark}

\subsection{Dirichlet form}
The Dirichlet form associated with $\{P^{F_\infty}_t\}_{t\geq 0}$ is given by
\begin{align*}
&\mcE^{F_\infty}(f,f)=\lim_{t\to 0}\frac{1}{t}\langle f-P_t^{F_\infty}f,f\rangle_{L^2(F_\infty,\mu_\infty)}\\
&\mcF^{F_\infty}=\{f\in L^2(F_\infty,\mu_\infty)~|~\mcE^{F_\infty}(f,f) \text{ exists and is finite}\},
\end{align*}
see e.\ g.~\cite[Definition 1.7.1]{BGL14}. 
As for the infinitesimal generator, also the Dirichlet forms $(\mcE^{F_i},\mcF^{F_i})$ with finite index $i$
may also be expressed in terms of cable systems/quantum graphs, c.f.~\cite[Remark 7 (i)]{AR18}. 
The main result in this paragraph is the \textit{generalized Mosco} convergence 
of the finite level Dirichlet forms to $(\mathcal{E}^{F_\infty},\mcF^{F_\infty})$. We recall from~\cite[Definition 2.11]{KS03}, see also~\cite[Definition 8.1]{CKK13}, that the sequence of quadratic forms 
$(\mcE^{F_i},\mcF^{F_i})$ Mosco converges to $(\mcE^{F_\infty},\mcF^{F_\infty})$ in the generalized sense if
\begin{enumerate}[leftmargin=3em,itemsep=1em]
\item[(M1)] For any sequence $\{f_i\}_{i\geq 0}$ with $f_i\in L^2(F_i,\mu_i)$ such that $\Phi_i^*f_i\xrightarrow{i\to\infty} f\in L^2(F_\infty,\mu_\infty)$ weakly in $L^2(F_\infty,\mu_\infty)$ it holds that
\begin{equation*}
\mcE^{F_\infty}(f,f)\leq\liminf_{i\to\infty}\mcE^{F_i}(f_i,f_i);
\end{equation*}
\item[(M2)] For any $f\in L^2(F_\infty,\mu_\infty)$ there exists $f_i\in L^2(F_i,\mu_i)$ such that $\Phi_i^*f_i\xrightarrow{i\to\infty} f$ strongly in $L^2(F_\infty,\mu_\infty)$ and 
\begin{equation*}
\limsup_{i\to\infty}\mcE^{F_i}(f_i,f_i)\leq \mcE^{F_\infty}(f,f).
\end{equation*}
\end{enumerate}
Combining the results from the previous paragraphs we can state several properties of the Dirichlet form.
\begin{theorem}\label{T:DF_props}
For the Dirichlet form $(\mcE^{F_\infty},\mcF^{F_\infty})$ associated with $\{P_t^{F_\infty}\}_{t\geq 0}$ it holds that
\begin{enumerate}[leftmargin=.25em,label={\rm (\roman*)},itemsep=.75em,wide=0em]
\item $(\mcE^{F_\infty},\mcF^{F_\infty})$ is the generalized Mosco limit of $\{(\mcE^{F_i},\mcF^{F_i})\}_{i\geq 0}$,
\item $\mcD_0$ is a core for $(\mcE^{F_\infty},\mcF^{F_\infty})$. 
\item\label{P:DF_props_3} For any $i\geq 1$ and $h\in \mcD_{F_i}$, $\displaystyle \mcE^{F_\infty}(\Phi_i^*h,\Phi_i^*h)=\mcE^{F_i}(h,h)$;
\item For any $f\in\mcF^{F_\infty}$ there exists $\{f_i\}_{i\geq 0}\subset \mcD_0$ such that 
\begin{equation}
\mcE^{F_\infty}(f,f)=\lim_{i\to\infty}\mcE^{F_i}(f_i,f_i).
\end{equation}
\item $(\mcE^{F_\infty},\mcF^{F_\infty})$ is local and regular.
\end{enumerate}
\end{theorem}
\begin{proof}
\begin{enumerate}[leftmargin=.25em,label=(\roman*),itemsep=.5em,wide=0em]
\item This follows from Lemma~\ref{L:Sg_strong_conv} and~\cite[Theorem 2.5]{Kim06}.
\item In view of Theorem~\ref{T:core_ig}, $\mcD_0$ is a core for $(\mcE^{F_\infty},\mcF^{F_\infty})$.
\item Notice that $\Phi_i^*h\in\mcD_0$. By Lemma~\ref{L:Intertwin}, for any $t>0$
\begin{align*}
\frac{1}{t}\langle \Phi_i^*h-P^{F_\infty}_t\Phi_i^*h,\Phi_i^*h\rangle_{L^2(F,\mu)}=\frac{1}{t}\langle \Phi_i^* h-\Phi_i^*P_t^{F_i}h,\Phi_i^*h\rangle_{L^2(F,\mu)}
=\frac{1}{t}\langle h- P_t^{F_i}h,h\rangle_{L^2(F_i,\mu_i)}.
\end{align*}
Letting $t\to 0$ we obtain $\mcE^{F_\infty}(\Phi_i^*h,\Phi_i^*h)=\mcE^{F_i}(h,h)$. 
\item By density, for any $f\in\mcF^{F_\infty}$ there exists a sequence $\{\Phi_i^* f_i\}_{i\geq 0}\subseteq\mcD_0\subset C(F_\infty)$ that approximates $f$ in $L^2(F_\infty,\mu_\infty)$, where $f_i\in C(F_i)$. Now, (i) and (ii) yield
\begin{equation}
\mcE^{F_i}(f_i,f_i)=\mcE^{F_\infty}(\Phi_i^*f_i,\Phi_i^*f_i)\xrightarrow{i\to\infty}\mcE^{F_\infty}(f,f).
\end{equation}
\item Since $\mcC_0\subseteq\mcF^{F_\infty}\cap C(F_\infty)$, the regularity follows from (ii). Thus, $(\mcE^{F_\infty},\mcF^{F_\infty})$ is local because $(\mcE^{F_i}, \mcF^{F_i})$ are, see also~\cite[Theorem 4.1]{ST12}.
\end{enumerate}
\end{proof}

\section{Estimates for the heat kernel}\label{S:contHK}
In this section we will use the expression of the heat kernel in~\eqref{E:HK_general_form} to obtain global heat kernel estimates and explicit bounds for the Lipschitz continuity of $p_t^{F_\infty}$ in $(F_\infty,d_\infty)$. In the particular case of regular diamonds with $n=j=2$, heat kernel estimates were obtained in~\cite[Theorem 4.7]{HK10} using ideas from~\cite{BK01} and exploiting the self-similarity of the space. The (joint) continuity of $p_t^{F_\infty}(x,y)$ may be derived using indirect arguments, see~\cite{AR18,HK10}, whereas the new estimates in Theorem~\ref{T:supHK} and Theorem~\ref{T:LipFinfty} provide a direct proof that gives information about the dependence of the bounds on the parameters.

\subsection{Uniform heat kernel bounds}
We start by showing that $p_t^{F_\infty}$ is uniformly bounded in $(F_\infty,d_\infty)$. The estimates will be applied in later sections to study related functional inequalities.

\begin{theorem}\label{T:supHK}
There exists $C_{\mathcal{N},\mathcal{J}}>0$ 
such that ,for any $t>0$,
\begin{equation}
\frac{1}{\sqrt{4\pi }}t^{-1/2}\leq \|p_t^{F_\infty}\|_\infty\leq\frac{1}{2\pi}+\frac{1}{\sqrt{4\pi }}t^{-1/2}+C_{\mathcal{N},\mathcal{J}}\, t^{-\frac{1}{2}(1+d(\ell_t^*))}, 
\end{equation}
where $d(\ell_t^*)=\frac{\log N_{\ell_t^*-1}}{\log J_{\ell_t^*-1}}$ and $\ell_t^*:=\inf\{\ell\geq 1\colon\, J_\ell^{-2}\leq t\}$.
\end{theorem}
From this result one readily deduces global estimates for short times.
\begin{corollary}\label{C:supHK}
There exists $C_{\mathcal{N},\mathcal{J}}>0$ such that
\begin{equation}\label{E:supHKshort}
\frac{1}{\sqrt{4\pi }}t^{-1/2}\leq \|p_t^{F_\infty}\|_\infty\leq C_{\mathcal{N},\mathcal{J}}\, t^{-\frac{1}{2}(1+d(\ell_t^*))}
\end{equation}
for any $t\in (0,1)$.
\end{corollary}

In particular, see Corollary~\ref{C:HKsup_regular}, the exponent on the right hand side of~\eqref{E:supHKshort} can be identified in the regular case with the \textit{spectral dimension} of $F_\infty$, which classically describes the short-time asymptotic behavior of the trace of the heat semigroup. 
\begin{proof}[Proof of Theorem~\ref{T:supHK}]
Let $t>0$ be fixed. Taking into consideration the convention $N_0=1$, the expression of $p_t^{F_\infty}$ given in~\eqref{E:HK_general_form} and Lemma~\ref{L:series_bound} yield
\begin{align}\label{E:supHK1}
p_t^{F_\infty}(x,y)&\leq p_t^{F_0}(\Phi_0(x),\Phi_0(y))+\frac{1}{\sqrt{\pi t}}\sum_{\ell=1}^\infty N_\ell\min\Big\{1,\frac{2}{(\pi J_\ell^2 t)^{1/2}}e^{-J_\ell^2t}\Big\}
\end{align}
for any $x,y\in X$. Define now $\ell_t^*:=\inf\{\ell\geq 1\,\colon\,J_\ell^{-2}\leq t\}$, so that 
\begin{equation}\label{E:supHK2}
J_{\ell_t^*}^{-2}\leq t <J_{\ell_t^*-1}^{-2}.
\end{equation}
To estimate the series in~\eqref{E:supHK1} we split it into two parts
\begin{equation}\label{E:supHK3}
\frac{1}{\sqrt{\pi t}}\sum_{\ell=1}^{\ell_t^*-1}N_\ell+\frac{2}{\sqrt{\pi t}}\sum_{\ell=\ell_t^*}^\infty\frac{N_\ell}{(J_\ell^2 t)^{1/2}}e^{-J_\ell^2t}=:\frac{1}{\sqrt{\pi t}}S_1+\frac{2}{\sqrt{\pi t}}S_2.
\end{equation}
For the first term we have
\begin{align}
S_1&=N_{\ell_t^*-1}\sum_{\ell=1}^{\ell_t^*-1}\frac{N_\ell}{N_{\ell_{t^*-1}}}=N_{\ell_t^*-1}\Big(1+\sum_{\ell=1}^{\ell_t^*-2}n_{\ell+1}^{-1}\cdots n_{{\ell_t^*-1}}^{-1}\Big)\notag\\
&\leq N_{\ell_t^*-1}\Big(1+\sum_{\ell=1}^{\ell_t^*-2}2^{-\ell_t^*+\ell+1}\Big) 
\leq N_{\ell_t^*-1}\sum_{k=0}^\infty 2^{-k}=2N_{\ell_t^*-1}.\label{E:supHK4}
\end{align}
For the second term in~\eqref{E:supHK3}, using~\eqref{E:supHK2} and the notation from Definition~\ref{D:prod_seq} we have
\begin{align}
S_2&=N_{\ell_t^*-1}\sum_{\ell=\ell_t^*}^\infty\frac{N_\ell}{N_{\ell_t^*-1}}\frac{J_{\ell_t^*}}{J_\ell(J_{\ell_t^*}^2 t)^{1/2}}\exp\Big\{-\frac{J^2_\ell}{J^2_{\ell_t^*}}J_{\ell_t^*}^2t\Big\}\notag \\
&\leq N_{\ell_t^*-1}\sum_{\ell=\ell_t^*}^\infty \frac{N_{\ell_t^*,\ell}}{J_{\ell_t^*+1,\ell}}e^{-J_{\ell_t^*+1,\ell}^2} 
\leq N_{\ell_t^*-1}\sum_{\ell=\ell_t^*}^\infty N_{\ell_t^*,\ell}e^{-J_{\ell_t^*+1,\ell}^2} \label{E:supHK5}
\end{align}
where the latter series converges and can be bounded independently of $t$ by Assumption~\ref{A:NJ_ass2}. Thus, there is a constant $C_{\mathcal{N},\mathcal{J}}>0$ that only depends on the parameter sequences $\mathcal{N},\mathcal{J}$ such that
\begin{equation*}
S_1+S_2\leq C_{\mathcal{N},\mathcal{J}}N_{\ell_t^*-1}=C_{\mathcal{N},\mathcal{J}}(J_{\ell_{t^*-1}}^2)^{\frac{1}{2}d(\ell_t^*)}\leq C_{\mathcal{N},\mathcal{J}} t^{-\frac{1}{2}d(\ell_t^*)},
\end{equation*}
where $d(\ell_t^*):=\frac{\log N_{\ell_{t^*-1}}}{\log J_{\ell_{t^*-1}}}$ and the last inequality is due to~\eqref{E:supHK2}.
Finally, applying Lemma~\ref{L:unif_F0} to the first term of~\eqref{E:supHK1} and the previous estimates for the second term yields
\begin{equation}
p_t^{F_\infty}(x,y)\leq\frac{1}{2\pi}+\frac{1}{\sqrt{4\pi }}t^{-1/2}+C_{\mathcal{N},\mathcal{J}} t^{-\frac{1}{2}(1+d(\ell_t^*))}.
\end{equation}
The lower bound readily follows from the expression~\eqref{E:HK_general_form} and~\eqref{E:PoissonFormula}. 
\end{proof}

\begin{remark}\label{R:series_bound}
Replacing Assumption~\ref{A:NJ_ass2} by~\ref{E:NJ_ass}, one concludes from~\eqref{E:supHK3} estimates of the type
\begin{equation*}
S_1\leq \ell_t^* t^{-\frac{1}{2}d(\ell_t^*)}\qquad\text{and}\qquad S_2\leq C_tt^{-\frac{1}{2}d(\ell_t^*)}
\end{equation*}
for some $C_t>0$ that bounds the series in~\eqref{E:supHK5}. In the same way as before, these would now provide
\begin{equation*}
p_t^{F_\infty}(x,y)\leq\frac{1}{2\pi}+\frac{1}{\sqrt{4\pi }}t^{-1/2}+t^{-\frac{1}{2}(1+d(\ell_t^*))}(C\ell_t^*+\tilde{C} C_t),
\end{equation*}
without further information about the dependence on $t$ of the constant $C_t$.
\end{remark}

\subsection*{Regular diamond fractals}
In the particular case of regular diamonds, since $n_i=n$ and $j_i=j$ for all $i\geq 1$, Assumption~\ref{A:NJ_ass2} can be checked by a direct computation since for any $k\geq 0$
\begin{align*}
\sum_{\ell=k}^\infty n^{\ell-k+1}j^{\ell-k}e^{-j^{2(\ell-k)}}
&\leq ne^{-1}+\int_0^\infty n^\xi j^\xi e^{-2\xi}d\xi
=ne^{-1}+\frac{1}{2\log j}\Gamma\Big(\frac{1}{2}\Big(1+\frac{\log n}{\log j}\Big)\Big).
\end{align*}
Similarly, we can compute explicitly the bound in~\eqref{E:supHK5} an get
\[
\sum_{\ell=k}^\infty n^{\ell-k+1}e^{-j^{2(\ell-k)}}\leq ne^{-1}+\int_0^\infty n^\xi e^{-2\xi}d\xi=ne^{-1}+\frac{1}{2\log j}\Gamma\Big(\frac{\log n}{2\log j}\Big).
\]
Finally, because $\frac{\log N_\ell}{\log J_\ell}=\frac{\log n}{\log j}$ for all $\ell\geq 1$, Theorem~\ref{T:supHK} provides the following global estimate of the sup-norm.
\begin{corollary}\label{C:HKsup_regular}
On a regular diamond fractal with parameters $n,j\geq 2$ there exists $C_{n,j}>0$ such that 
\begin{equation}
\frac{1}{\sqrt{4\pi }}t^{-1/2}\leq \|p^{F_\infty}_t\|\leq \frac{1}{2\pi}+\frac{1}{\sqrt{4\pi }}t^{-1/2}+C_{n,j}\,t^{-\frac{1}{2}\big(1+\frac{\log n}{\log j}\big)}.
\end{equation}
\end{corollary}
In particular, $1+\frac{\log n}{\log j}$ coincides with the spectral dimension $d_S$ of $F_\infty$, see e.g.~\cite[Theorem A.2]{KL93}. 
In this case, it coincides with the Hausdorff dimension $d_H$, in agreement with the observation that the walk dimension of a diamond fractal is $d_w=\frac{2d_H}{d_S}=2$.

\subsection{Continuity estimates}
The recursive nature of the underlying space is also reflected in the proof of the continuity of the heat kernel. In particular, the bound of the Lipschitz constant relies on a recursive argument, for which the case $i=1$ serves both as guideline and first induction step. The different pair-point configurations for that level, summarized in Figure~\ref{F:GeneralPairPoint}, will be analyzed by means of standard estimates recorded in Appendix~\ref{S:Apx_HKcircle}.

\begin{theorem}\label{T:LipFinfty}
For any $t>0$, the heat kernel $p^{F_\infty}_t\colon F_\infty\times F_\infty\to [0,\infty)$ is Lipschitz continuous in $(F_\infty,d_\infty)$ 
and satisfies for any $x,y_1,y_2\in F_\infty$ 
\begin{equation}\label{E:LipFinfty}
|p_t^{F_\infty}(x,y_1)-p_t^{F_\infty}(x,y_2)|\leq  C t^{-1-\frac{1}{2}d(\ell_t^*)} d_\infty(y_1,y_2),
\end{equation}
where $d(\ell_t^*)=\frac{\log N_{\ell_*^*-1}}{\log J_{\ell_*^*-1}}$, $\ell_t^*:=\inf\{\ell\geq 1\,\colon\,J_\ell^{-2}\leq t\}$ and some constant $C>0$ depending on the parameter sequences $\mathcal{N}$, $\mathcal{J}$.
\end{theorem}

\begin{proof}
Since $p_t^{F_i}(x,y)$ converges uniformly to $p^{F_\infty}_t(x,y)$, see~\cite[Remark 8]{AR18}, the latter is continuous and its Lipschitz constant $C_L(t)$ may be bounded by taking the limit $i\to\infty$ in Proposition~\ref{T:Lip_Fi}, which leads to
\begin{equation}\label{E:LipFinfty_series}
C_L(t)\leq \frac{2}{\pi}\sum_{\ell=0}^\infty N_{\ell}\Big(J_\ell^2+\frac{1}{2t}\Big)e^{-J_\ell^2t}.
\end{equation}
To estimate this series, we consider again $\ell_t^*:=\inf\{\ell\geq 1\,\colon\,J_\ell^{-2}\leq t\}$, so that $J_{\ell_t^*}^{-2}\leq t <J_{\ell_t^*-1}^{-2}$ and decompose the series in~\eqref{E:LipFinfty_series} into the two terms
\begin{equation*}
\sum_{\ell=0}^{\ell_t^*-1} N_{\ell}\Big(J_\ell^2+\frac{1}{2t}\Big)e^{-J_\ell^2t}+\sum_{\ell=\ell_t^*}^\infty N_{\ell}\Big(J_\ell^2+\frac{1}{2t}\Big)e^{-J_\ell^2t}=:S_1+S_2.
\end{equation*}
For the first, with the notation of Definition~\ref{D:prod_seq}, analogous computations to those in~\eqref{E:supHK4} yield
\begin{equation*}
S_1\leq\frac{3}{2t}\sum_{\ell=0}^{\ell_t^*-1}N_\ell e^{-J_{\ell_t^*}^2t}\leq\frac{3}{2t}N_{\ell_t^*-1}\sum_{\ell=0}^{\ell_t^*-1}N_{\ell+1,\ell^*_t-1}^{-1}
\leq \frac{3}{2t}N_{\ell_t^*-1}\sum_{\ell=0}^{\ell_t^*-1}2^{-(\ell_t^*-\ell-1)}
\leq 
\frac{3}{t}N_{\ell_t^*-1}
\end{equation*}
For the second term, using $J_{\ell_t^*}^{-2}\leq t <J_{\ell_t^*-1}^{-2}$ we have
\begin{align}
S_2&\leq \frac{3}{2t}\sum_{\ell=\ell_t^*}^\infty N_\ell J_\ell^2t e^{-J_\ell^2 t}=\frac{3}{2t}N_{\ell_t^*-1}J_{\ell_t^*-1}^2 t \sum_{\ell=\ell_t^*}^\infty N_{\ell_t^*,\ell}J_{\ell_t^*,\ell}^2e^{-J_\ell^2 t}\notag\\
&\leq \frac{3}{2t}N_{\ell_t^*-1} \sum_{\ell=\ell_t^*}^\infty N_{\ell_t^*,\ell}J_{\ell_t^*,\ell}^2e^{-J_{\ell_t^*+1,\ell}^2}.\label{E:LipFinfty_01}
\end{align}
By virtue of Assumption~\ref{A:NJ_ass2} the latter series is bounded independently of $t$. Combining the two previous estimates and setting $d(\ell_t^*):=\frac{\log N_{\ell_{t^*-1}}}{\log J_{\ell_{t^*-1}}}$ it follows 
that there is $C>0$ such that
\begin{equation*}
S_1+S_2\leq Ct^{-1}N_{\ell_t^*-1}=Ct^{-1}N_{\ell_t^*-1}=Ct^{-1}(J_{\ell_{t^*-1}}^2)^{\frac{1}{2}d(\ell_t^*)}\leq C t^{-1-\frac{1}{2}d(\ell_t^*)}.
\end{equation*}
\end{proof}

\begin{proposition}\label{T:Lip_Fi}
For any $t>0$, the heat kernel $p_t^{F_i}$ is Lipschitz continuous in $(F_i,d_i)$ and
\begin{equation}\label{E:LipFi}
|p_t^{F_i}(x,y_1)-p_t^{F_i}(x,y_2)|
\leq \frac{2}{\pi}\sum_{\ell=0}^i N_{\ell}\Big(J_\ell^2+\frac{1}{2t}\Big)e^{-J_\ell^2t}
d(y_1,y_2).
\end{equation}
\end{proposition}

This result is proved by induction, whose first step is presented below separately. For the ease of the notation and to remain consistent with the one appearing in~\cite{AR18}, we will set $\theta_x:=\phi_{i0}(x)$ for $x\in F_i$ and $\theta_x:=\Phi_0(x)$ for $x\in F_\infty$.

\begin{proposition}\label{P:Lip_F1}
For each $t>0$, the heat kernel $p_t^{F_1}$ is Lipschitz continuous in $(F_1,d_1)$ and 
\begin{equation*}
|p_t^{F_1}(x,y_1)-p_t^{F_1}(x,y_2)|
\leq \frac{2}{\pi}\Big[\Big(1+\frac{1}{2t}\Big)e^{-t}+n_1\Big(j^2_1+\frac{1}{2t}\Big)e^{-j_1^2t}\Big]
d(y_1,y_2).
\end{equation*}
\end{proposition}
\begin{proof}
By virtue of Lemma~\ref{L:chain_cond} and the triangle inequality, a configuration as in Figure~\ref{F:PP_generic} will be covered by any of the basic cases shown in Figures~\ref{F:PPcase_a} through~\ref{F:PPcase_c}.

\begin{figure}[H]
\begin{subfigure}[b]{.4\textwidth}
\centering
\begin{tikzpicture}
\coordinate[label=right:{$x$}] (x) at ($(30:1.15)$);
\fill (x) circle (1.5pt);
\coordinate[label=above:{$y_1$}] (y1) at ($(90:1)$);
\fill (y1) circle (1.5pt);
\coordinate[label=left:{$y_2$}] (y2) at ($(320:.75)$);
\fill (y2) circle (1.5pt);
\foreach \a in {0,60,120,180,240,300}{
\draw ($(\a:1)$) to[out=90-55+\a,in=15+\a] ($(60+\a:1)$);
\draw ($(\a:1)$) to[out=90+\a,in=330+\a] ($(60+\a:1)$);
\draw ($(\a:1)$) to[out=90+65+\a,in=-90+5+\a] ($(60+\a:1)$);
}
\end{tikzpicture}
\caption{$y_1$ and $y_2$ belong to bundles different than the one of $x$.}
\label{F:PPcase_a}
\end{subfigure}
\hspace*{1em}
\begin{subfigure}[b]{.4\textwidth}
\centering
\begin{tikzpicture}
\coordinate[label=right:{$x$}] (x) at ($(30:1.15)$);
\fill (x) circle (1.5pt);
\coordinate[label=above:{$y_1$}] (y1) at ($(40:1)$);
\fill (y1) circle (1.5pt);
\coordinate[label=left:{$y_2$}] (y2) at ($(20:.75)$);
\fill (y2) circle (1.5pt);
\foreach \a in {0,60,120,180,240,300}{
\draw ($(\a:1)$) to[out=90-55+\a,in=15+\a] ($(60+\a:1)$);
\draw ($(\a:1)$) to[out=90+\a,in=330+\a] ($(60+\a:1)$);
\draw ($(\a:1)$) to[out=90+65+\a,in=-90+5+\a] ($(60+\a:1)$);
}
\end{tikzpicture}
\caption{$y_1$ and $y_2$ belong different branches in the bundle of $x$.}
\label{F:PPcase_b}
\end{subfigure}
\end{figure}
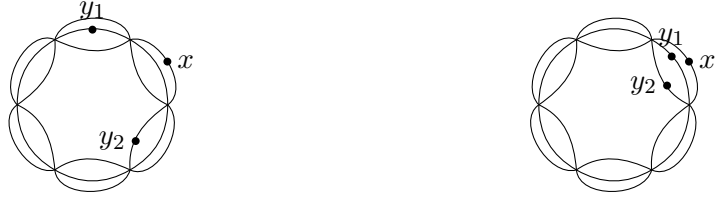
\begin{figure}[H]\ContinuedFloat
\begin{subfigure}[b]{.4\textwidth}
\centering
\begin{tikzpicture}
\coordinate[label=right:{$x$}] (x) at ($(30:1.15)$);
\fill (x) circle (1.5pt);
\coordinate[label=above:{$y_1$}] (y1) at ($(40:1.15)$);
\fill (y1) circle (1.5pt);
\coordinate[label=right:{$y_2$}] (y2) at ($(15:1.15)$);
\fill (y2) circle (1.5pt);
\foreach \a in {0,60,120,180,240,300}{
\draw ($(\a:1)$) to[out=90-55+\a,in=15+\a] ($(60+\a:1)$);
\draw ($(\a:1)$) to[out=90+\a,in=330+\a] ($(60+\a:1)$);
\draw ($(\a:1)$) to[out=90+65+\a,in=-90+5+\a] ($(60+\a:1)$);
}
\end{tikzpicture}
\caption{Both $y_1$ and $y_2$ belong to the same branch as $x$.}
\label{F:PPcase_c}
\end{subfigure}
\hspace*{1em}
\begin{subfigure}[b]{.4\textwidth}
\centering
\begin{tikzpicture}
\coordinate[label=right:{$x$}] (x) at ($(30:1.15)$);
\fill (x) circle (1.5pt);
\coordinate[label=left:{$y_1$}] (y1) at ($(20:.75)$);
\fill (y1) circle (1.5pt);
\coordinate[label=right:{$y_2$}] (y2) at ($(220:.75)$);
\fill (y2) circle (1.5pt);
\foreach \a in {0,60,120,180,240,300}{
\draw ($(\a:1)$) to[out=90-55+\a,in=15+\a] ($(60+\a:1)$);
\draw ($(\a:1)$) to[out=90+\a,in=330+\a] ($(60+\a:1)$);
\draw ($(\a:1)$) to[out=90+65+\a,in=-90+5+\a] ($(60+\a:1)$);
}
\coordinate[label=right:{$z_0$}] (z1) at ($(0:1)$);
\fill[color=blue] (z1) circle (1.5pt);
\end{tikzpicture}
\caption{$y_1$ and $y_2$ belong to different bundles, and $y_2$ to the bundle of $x$.}
\label{F:PP_generic}
\end{subfigure}
\caption{Basic pair-point configurations.}
\label{F:GeneralPairPoint}
\end{figure}

\begin{enumerate}[leftmargin=.25em,label=(\alph*),itemsep=1em,wide=0em]
\item Each $y_1$ and $y_2$ belong to bundles that are different than the one $x$ belongs to, see Figure~\ref{F:PPcase_a}. In view of the expression of $p_t^{F_1}(x,y)$, Lemma~\ref{L:HK_formulations} and Lemma~\ref{L:dist} we have (recall $\theta_x:=\phi_{i0}(x)$)
\begin{align*}
|p_t^{F_1}(x,y_1)-p_t^{F_1}(x,y_2)|&=|p_t^{F_0}(\theta_x,\theta_{y_1})-p_t^{F_0}(\theta_x,\theta_{y_2})|\\
&\leq\frac{1}{\pi}\sum_{k=1}^\infty e^{-k^2t}|\cos\big(k(\theta_{y_1}-\theta_x)\big)-\cos\big(k(\theta_{y_2}-\theta_x)\big)|\\
&\leq\frac{1}{\pi}\sum_{k=1}^\infty e^{-k^2t}k|\theta_{y_1}-\theta_{y_2}|
\leq\Big(\frac{1}{\pi}\sum_{k=1}^\infty k e^{-k^2t}\Big) d_1(y_1,y_2)\\%
&\leq \frac{1}{\pi}\Big(e^{-t}+\int_1^\infty\xi e^{-\xi^2t}d\xi\Big)d_1(y_1,y_2)
=\frac{1}{\pi}\Big(1+\frac{1}{2t}\Big)e^{-t}d_1(y_1,y_2).
\end{align*}
\item Each of $y_1$ and $y_2$ belong to a different branch than $x$ but from the same bundle as $x$, see Figure~\ref{F:PPcase_b}. Then, 
\begin{equation*}
|p_t^{F_1}(x,y_1)-p_t^{F_1}(x,y_2)|\leq|p_t^{F_0}(\theta_x,\theta_{y_1})-p_t^{F_0}(\theta_x,\theta_{y_2})|+|p_t^{[0,L_1]_D}(\theta_x,\theta_{y_1})-p_t^{[0,L_1]_D}(\theta_x,\theta_{y_2})|.
\end{equation*}
The first term can be estimated as in the case (i). For the second term we have
\begin{align*}
|p_t^{[0,L_1]_D}(\theta_x,\theta_{y_1})-p_t^{[0,L_1]_D}(\theta_x,\theta_{y_2})|&
\leq\frac{2}{L_1}\sum_{k=1}^\infty e^{-\frac{k^2\pi^2}{L_1^2}t}\Big|\sin\Big(\frac{k\pi\theta_{y_1}}{L_1}\Big)-\sin\Big(\frac{k\pi\theta_{y_2}}{L_1}\Big)\Big|\\
&\leq \Big(\frac{2j_1^2}{\pi}\sum_{k=1}^\infty k  e^{-k^2j_1^2t}\Big)d_0(\phi_1(y_1),\phi_1(y_2))\\
&\leq \frac{2}{\pi}\Big(j^2_1e^{-j_1^2t}+\int_1^\infty j^2_1\xi  e^{-\xi^2j_1^2t}d\xi\Big)d_1(y_1,y_2)\\
&= \frac{2}{\pi}\Big(j^2_1e^{-j_1^2t}+\frac{1}{2 t}e^{-j_1^2t}\Big)d_1(y_1,y_2).
\end{align*}
\item Finally, if both $y_1$ and $y_2$ belong to the same branch as $x$, see Figure~\ref{F:PPcase_c}, then
\begin{equation*}
|p_t^{F_1}(x,y_1)-p_t^{F_1}(x,y_2)|\leq|p_t^{F_0}(\theta_x,\theta_{y_1})-p_t^{F_0}(\theta_x,\theta_{y_2})|+(n_1-1)|p_t^{[0,L_1]_D}(\theta_x,\theta_{y_1})-p_t^{[0,L_1]_D}(\theta_x,\theta_{y_2})|
\end{equation*}
which reduces to the previous case with an extra factor $(n_1-1)$ in the second summand.
\end{enumerate}
\end{proof}
The previous proposition serves now both as proof schema and first step to show the corresponding estimate for the Lipschitz constant in a generic finite level.

\begin{proof}[Proof of Proposition~\ref{T:Lip_Fi}]
We argue by induction, the case $i=1$ being Proposition~\ref{P:Lip_F1}. By virtue of the triangle inequality it suffices to prove~\eqref{E:LipFi} 
for any fixed $t>0$, $x\in F_i$ and $y_1,y_2$ such that $(x,y_1)$ and $(x,y_2)$ have the same pair-point configuration. Indeed, if $(x,y_1)$ and $(x,y_2)$ are of different type, we can always choose $z_0\in B_i$ such that $d_i(y_1,y_2)=d_i(y_1,z_0)+d_i(z_0,y_2)$ and consider $z_0$ each time as belonging to the particular cell for which $(x,z_0)$ and $(x,y_1)$ are of the same type, and $(x,z_0)$ and $(x,y_2)$ are of the same type, see Figure~\ref{F:PP_generic}. Let us thus assume that the Lipschitz constant for the $(i-1)$approximation, $C_L^{(i-1)}(t)$, can be bounded by the corresponding expression from~\eqref{E:LipFi}.

\begin{enumerate}[leftmargin=.25em,label=(\roman*),itemsep=1em,wide=0em]
\item If $(x,y)$ and $(x,y')$ are both of the first type, see Figure~\ref{F:PPcase_a}, the expression of the heat kernel and Lemma~\ref{L:dist} yield
\begin{align*}
|p_t^{F_i}(x,y_1)-p_t^{F_i}(x,y_2)|&=|p_t^{F_{i-1}}(\phi_i(x_,\phi_i(y_1))-p_t^{F_{i-1}}(\phi_i(x),\phi_i(y_2))|\\
&\leq C_L^{(i-1)}(t)d_{i-1}(\phi_i(y_1),\phi_i(y_2))\leq C_L^{(i-1)}(t)d_i(y_1,y_2).
\end{align*}
\item If $(x,y_1)$ and $(x,y_2)$ are both of the second type, see Figure~\ref{F:PPcase_b}, then
\begin{multline*}
|p_t^{F_i}(x,y_1)-p_t^{F_i}(x,y_2)|\leq |p_t^{F_{i-1}}(\phi_i(x,\phi_i(y_1))-p_t^{F_{i-1}}(\phi_i(x),\phi_i(y_2))|\\
+N_{i-1}|p_t^{[0,L_i]_D}(\theta_x,\theta_{y_1})-p_t^{[0,L_i]_D}(\theta_x,\theta_{y_2})|.
\end{multline*}
The first term can be estimated as in the case (i). For the second term we have from the proof of the case $i=1$ in Proposition~\ref{P:Lip_F1} (substituting $L_1$ by $L_i$) and Lemma~\ref{L:dist} that
\begin{align*}
|p_t^{[0,L_i]_D}(\theta_x,\theta_{y_1})&-p_t^{[0,L_i]_D}(\theta_x,\theta_{y_2})|
\leq \frac{2}{\pi}\Big(J_i^2e^{-J_i^2t}+\frac{1}{2 t}e^{-J_i^2t}\Big)d_i(y_1,y_2).
\end{align*}
\item The case when $(x,y_1)$ and $(x,y_2)$ are both of the third type, see Figure~\ref{F:PPcase_c}, reduces to the previous case with an extra factor $(n_i-1)$. 
\end{enumerate}
Putting all estimates together and using the induction hypothesis we have
\begin{align*}
|p_t^{F_i}(x,y_1)-p_t^{F_i}(x,y_2)| 
&\leq \Big(C_L^{(i-1)}(t)+\frac{2}{\pi}N_i\Big(J_i^2e^{-J_i^2t}+\frac{1}{2 t}e^{-J_i^2t}\Big)\Big)d_i(y_1,y_2)\\
&=\Big(\frac{2}{\pi}\sum_{\ell=0}^iN_\ell\Big(J_\ell^2+\frac{1}{2 t}\Big)e^{-J_\ell^2t}\Big)d_i(y_1,y_2)
\end{align*}
as we wanted to prove.
\end{proof}
\begin{remark}\label{R:bdd series}
The Lipschitz constant of Theorem~\ref{T:LipFinfty} is bounded by the series~\eqref{E:LipFinfty_series}. The convergence of this series is already guaranteed under the weaker condition~\eqref{E:NJ_ass}.
However, at this level of generality, again little more may be said about the behavior of the series as a function of $t$.
\end{remark}

\section{Continuity estimates of the heat semigroup}\label{S:contHS}
The aim of this section is to study the regularity of the heat semigroup $\{P_t^{F_\infty}\}_{t\geq 0}$ as a means to describe the geometry of $F_\infty$ along the lines of the so-called \textit{weak Barky-\'Emery curvature condition} from~\cite{ABCRST3}. This condition reads
\begin{equation}\label{E:BE_kappa}
|P^{F_\infty}_tf(x)-P^{F_\infty}_tf(y)|\leq C\frac{d_\infty(x,y)^\kappa}{t^{\kappa/d_w}}\|f\|_\infty
\end{equation}
for all $f\in L^\infty(F_\infty,\mu_\infty)$, where $\kappa>0$ denotes a (curvature) parameter and $d_w>0$ the \textit{walk dimension} of the space. We refer to~\cite{ABCRST3,ABCRST4} for further details and functional analytic consequences in the context of Dirichlet spaces with various heat kernel estimates. 
On each approximation level, Proposition~\ref{T:BE_Fi} shows that the condition~\eqref{E:BE_kappa} is satisfied with $\kappa=1$ and $d_w=2$. While the latter is expected because each $F_i$ is a one-dimensional object, the situation in the limit is less clear. The estimate in Theorem~\ref{T:BE_Finfty} reveals in concrete computations such as Corollary~\ref{C:BE_interval} a logarithmic correction that is also observed in diffusion processes with multifractal structures, see e.g.~\cite{BK01}. However, it is still an open question, whether this time dependence is actually optimal.

\begin{theorem}\label{T:BE_Finfty}
For any $t>0$, there exists a constant $C>0$ such that
\begin{equation}\label{E:BE_Finfty}
|P^{F_\infty}_tf(x)-P^{F_\infty}_t f(y)|\leq \frac{C}{\sqrt{t}}(1+\ell_t^*) \,d_\infty(x,y)\,\|f\|_\infty ,
\end{equation}
for any $f\in L^\infty(F_\infty,\mu_\infty)$ and $x,y\in F_\infty$, where $\ell_t^*:=\inf\{\ell\geq 1\colon\, J_\ell^{-2}\leq t\}$.
\end{theorem}


The proof of Theorem~\ref{T:BE_Finfty} is presented in detail at the end of the section and it relies as in Theorem~\ref{T:LipFinfty} on the corresponding result for the approximating spaces $F_i$. 
In contrast to the latter, the ``chain property'' from Lemma~\ref{L:chain_cond} turns out crucial to reduce the analysis of pair-point configurations to the case of pairs that belong to the same branch. 
%
%
We continue using the notation $\theta_x:=\phi_{i0}(x)$ for any $x\in F_i$, $i\geq 1$ and $\theta_x:=\Phi_0(x)$ if $x\in F_\infty$.

\begin{remark}
It is worthwhile pointing out that, in view of the estimate~\eqref{E:BE_infty_01}, the constant $C$ appears to be independent of the parameter sequence $\mcN$ that gives the number of copies (``parallel universes'', as named in~\cite{BE04}) of a level that give rise to the next.
\end{remark}
\subsection{Key lemma} 
The following estimate will be applied several times throughout the proof of the steps that yield the main result. To be consistent with the notation in~\cite{AR18}, we write $L_i:=\pi/J_i$.
%

\begin{lemma}\label{L:basic_BE}
Let $i\geq 0$. For any $x,y\in F_i$ and $t>0$,
\begin{equation}\label{E:basic_BE}
\int_0^{L_i}J_i|p^{F_0}_{J_i^2t}(J_i\rho,J_i\theta_x)-p^{F_0}_{J_i^2t}(J_i\rho,J_i\theta_y)|\,d\rho\leq \min\Big\{\frac{2}{\sqrt{\pi t}},\Big(J_i+\frac{1}{2J_it}\Big)e^{-J_i^2t}\Big\}|\theta_x-\theta_y|.
\end{equation}
\end{lemma}
In view of the relation~\eqref{E:Dir_HKvsF0}, Lemma~\ref{L:basic_BE} readily implies another useful inequality.
\begin{corollary}\label{C:BE_interval}
Let $i\geq 1$. For any $f\in L^\infty(F_i)$, 
\begin{equation*}
\int_0^{L_i}|(p_t^{[0,L_i]_D}(\rho,\theta_x)-p^{[0,L_i]_D}_t(\rho,\theta_y))f(\rho)|\,d\rho\\
\leq 2\min\Big\{\frac{2}{\sqrt{\pi t}},\Big(J_i+\frac{1}{2J_it}\Big)e^{-J_i^2t}\Big\}\|f\|_\infty|\theta_x-\theta_y|.
\end{equation*}
\end{corollary}

\begin{proof}[Proof of Lemma~\ref{L:basic_BE}]
The strategy consists in estimating the left hand side of~\eqref{E:basic_BE}, which we denote by $I$, with both representations of the heat kernel $p_t^{F_0}(x,y)$ given in~\eqref{E:PoissonFormula}. 
\begin{enumerate}[leftmargin=.25em,label=(\alph*),itemsep=1em,wide=0em]
\item Using the first representation in~\eqref{E:PoissonFormula},
\begin{equation}
I:=\frac{1}{\sqrt{4\pi t}}\int_0^{L_i}\Big|\sum_{k\in\mbbZ}\Big(e^{-\frac{(J_i\rho-J_i\theta_x-2\pi k)^2}{4J_i^2t}}-e^{-\frac{(J_i\rho-J_i\theta_y-2\pi k)^2}{4J_i^2t}}\Big)\Big|\,d\rho.
\end{equation}
Applying the triangle inequality,
\begin{align*}
&\sqrt{4\pi t} I\leq \int_0^{L_i}\sum_{k\in\mbbZ}\Big| e^{-\frac{(J_i\rho-J_i\theta_x-2\pi k)^2}{4J_i^2t}}-e^{-\frac{(J_i\rho-J_i\theta_y-2\pi k)^2}{4J_i^2t}}\Big|\,d\rho\nonumber\\
&=\int_0^{L_i}\Big| e^{-\frac{(J_i\rho-J_i\theta_x)^2}{4J_i^2t}}-e^{-\frac{(J_i\rho-J_i\theta_y)^2}{4J_i^2t}}\Big|\,d\rho
+ \int_0^{L_i}\sum_{k\geq 1}\Big| e^{-\frac{(J_i\rho-J_i\theta_x-2\pi k)^2}{4J_i^2t}}-e^{-\frac{(J_i\rho-J_i\theta_y-2\pi k)^2}{4J_i^2t}}\Big|\,d\rho\nonumber\\
&+\int_0^{L_i}\sum_{k\geq 1}\Big| e^{-\frac{(J_i\rho-J_i\theta_x+2\pi k)^2}{4J_i^2t}}-e^{-\frac{(J_i\rho-J_i\theta_y+2\pi k)^2}{4J_i^2t}}\Big|\,d\rho=:I_1+I_2+I_3.
\end{align*}
Without loss of generality, let us assume that $\theta_x\leq \theta_y$
. For the first integral term we have
\begin{align*}
I_1&\leq \int_0^{L_i}\Big|\int_{\theta_x}^{\theta_y}\frac{J_i\rho -J_i\tilde{\rho}}{2J_it}e^{-\frac{(J_i\rho-J_i\tilde{\rho})^2}{4J_i^2t}}d\tilde{\rho}\Big|\,d\rho\leq \int_0^{L_i}\int_{\theta_x}^{\theta_y}\frac{|\rho -\tilde{\rho}|}{2t}e^{-\frac{(\rho-\tilde{\rho})^2}{4t}}d\tilde{\rho}\,d\rho\\
&=\int_{\theta_x}^{\theta_y}\int_0^{\tilde{\rho}}-\frac{\rho -\tilde{\rho}}{2t}e^{-\frac{(\rho-\tilde{\rho})^2}{4t}}d\rho\,d\tilde{\rho}+\int_{\theta_x}^{\theta_y}\int_{\tilde{\rho}}^{L_i}\frac{\rho -\tilde{\rho}}{2t}e^{-\frac{(\rho-\tilde{\rho})^2}{4t}}d\rho\,d\tilde{\rho}\\
&=\int_{\theta_x}^{\theta_y}\Big[e^{-\frac{(\rho-\tilde{\rho})^2}{4t}}\Big]_0^{\tilde{\rho}} d\tilde{\rho}+\int_{\theta_x}^{\theta_y}\Big[-e^{-\frac{(\rho-\tilde{\rho})^2}{4t}}\Big]_{\tilde{\rho}}^{L_i} d\tilde{\rho}\\
&=\int_{\theta_x}^{\theta_y}\Big(1-e^{-\frac{\tilde{\rho}^2}{4t}}\Big)d\tilde{\rho}+\int_{\theta_x}^{\theta_y}\Big(1-e^{-\frac{(L_i-\tilde{\rho})^2}{4t}}\Big) d\tilde{\rho}\\
&\leq 2|\theta_y-\theta_x|.
\end{align*}
Moreover, recall that $\theta_x,\theta_y\in [0,L_i)$ for any $x,y\in F_i$. Hence, for any $k\geq 1$, $\rho\in [\theta_x, \theta_y]$ and $\tilde{\rho}\in [0,L_i)$, the quantity $\rho-\tilde{\rho}-2kL_i\leq L_i-2kL_i$ is nonpositive. Thus,
\begin{align*}
I_2&\leq \int_0^{L_i} \sum_{k\geq 1}\Big|\int_{\theta_x}^{\theta_y}\frac{2(J_i\rho -J_i\tilde{\rho}-2\pi k)}{4J_it}e^{-\frac{(J_i\rho-J_i\tilde{\rho}-2\pi k)^2}{4J_i^2t}}d\tilde{\rho}\Big|\,d\rho\\
&\leq \int_{\theta_x}^{\theta_y}\sum_{k\geq 1}\int_0^{L_i}\frac{|\rho -\tilde{\rho}-2\pi k/J_i|}{2t}e^{-\frac{(\rho-\tilde{\rho}-2\pi k/J_i)^2}{4t}}d\tilde{\rho}\,d\rho\\
&=\int_{\theta_x}^{\theta_y}\sum_{k\geq 1}\int_0^{L_i}-\frac{\rho -\tilde{\rho}-2kL_i }{2t}e^{-\frac{(\rho-\tilde{\rho}-2kL_i )^2}{4t}}d\tilde{\rho}\,d\rho\\
&\leq \int_{\theta_x}^{\theta_y}\int_0^{L_i}\int_0^\infty-\frac{\rho -\tilde{\rho}-\xi }{2t}e^{-\frac{(\rho-\tilde{\rho}-\xi )^2}{4t}}d\tilde{\rho}\,d\rho\\
&=\frac{1}{L_i}\int_{\theta_x}^{\theta_y}\int_0^{L_i}e^{-\frac{(\rho-\tilde{\rho})^2}{4t}}d\tilde{\rho}\,d\rho
\leq |\theta_x-\theta_y|
\end{align*}
Analogously, because $\rho-\tilde{\rho}+2kL_i\geq L_i+2kL_i>0$ for any $k\geq 1$, we obtain
\begin{equation*}
I_3\leq \int_{\theta_x}^{\theta_y}\int_0^{L_i}\sum_{k\geq 1}\frac{\rho -\tilde{\rho}+2kL_i }{2t}e^{-\frac{(\rho-\tilde{\rho}+2kL_i )^2}{4t}}d\tilde{\rho}\,d\rho\leq |\theta_x-\theta_y|.
\end{equation*}
Adding up these estimates leads to
\begin{equation}\label{E:basic_BE_help02}
I\leq \frac{1}{\sqrt{4\pi t}}(I_1+I_2+I_3)=\frac{2}{\sqrt{\pi t}}|\theta_x-\theta_y|.
\end{equation}
\item Using the second representation of the heat kernel $p_t^{F_0}(x,y)$ in~\eqref{E:PoissonFormula} we write
\begin{equation*}
I:=\frac{J_i}{\pi}\int_0^{L_i}\Big|\sum_{k\geq 1}e^{-k^2J_i^2t}\big(\cos(kJ_i(\theta_x-\rho))-\cos(kJ_i(\theta_y-\rho))\big)\Big|\,d\rho.
\end{equation*}
Then, since $L_i=\pi/J_i$ we have
\begin{align}
I&\leq \frac{J_i}{\pi}\int_0^{L_i}\sum_{k\geq 1}e^{-k^2J_i^2t}|\cos(kJ_i(\theta_x-\rho))-\cos(kJ_i(\theta_y-\rho))|\,d\rho\nonumber\\
&\leq \frac{J_i}{\pi}\int_0^{L_i}\sum_{k\geq 1}e^{-k^2J_i^2t}kJ_i |\theta_x-\theta_y|\,d\rho= \frac{J_iL_i}{\pi} |\theta_x-\theta_y|\sum_{k\geq 1}e^{-k^2J_i^2t}kJ_i\nonumber\\
&\leq |\theta_x-\theta_y|\Big(J_ie^{-J_i^2t}+\int_1^\infty J_i\xi e^{-\xi^2J_i^2t}d\xi\Big)\nonumber\\
&=\Big(J_ie^{-J_i^2t}+\frac{1}{2J_i t}e^{-J_i^2 t}\Big)|\theta_x-\theta_y|=e^{-J_i^2t}\Big(J_i+\frac{1}{2J_it}\Big)|\theta_x-\theta_y|.\label{E:basic_BE_help03}
\end{align}
\end{enumerate}
The assertion now follows from~\eqref{E:basic_BE_help02} and~\eqref{E:basic_BE_help03}.
\end{proof}

\subsection{First approximation level}
The weak Bakry-\'Emery condition~\eqref{E:BE_kappa} with $\kappa=1$ and $d_w=2$ is obtained on each finite approximation $F_i$ by an inductive argument and this paragraph is devoted to the first induction step. The ``chain property'' from Lemma~\ref{L:chain_cond} will allow us to reduce the analysis to pairs of points $x,y\in F_1$ that belong to the same branch, see Figure~\ref{F:PPcase_c}. Any notation appearing in the proof for the first time follows~\cite{AR18} and is briefly recalled in Appendix~\ref{S:Apx_AR18}.
\begin{proposition}\label{P:BE_F1}
For any $t>0$, $f\in L^\infty(F_1)$ and $x,y\in F_1$,
\begin{equation}\label{E:BE_F1}
|P_t^{F_1}f(x)-P_t^{F_1}f(y)|\leq C_1(t)\,d_1(x,y)\|f\|_\infty ,
\end{equation}
where
\begin{equation}\label{E:BE_C1}
C_1(t)\leq 2\Big(\min\Big\{\frac{2}{\sqrt{\pi t}},\Big(1+\frac{1}{2t}\Big)e^{-t}\Big\}+\min\Big\{\frac{2}{\sqrt{\pi t}},\Big(j_1+\frac{1}{2j_1 t}\Big)e^{-j_1^2t}\Big\}\Big).
\end{equation}
\end{proposition}

\begin{proof}
By virtue of the triangle inequality and~\cite[Proposition 3]{AR18}, see also Lemma~\ref{L:HSFA01}, we have (recall $\theta_x:=\phi_{i0}(x)$)
\begin{multline}\label{E:BEdelta_F1_help02}
|P_t^{F_1}f(x)-P_t^{F_1}f(y)|\leq|P_t^{F_0}(\mcI_1 f)(\theta_x)-P_t^{F_0}(\mcI_1 f)(\theta_y)|\\
+|P_t^{[0,L_1]_D}(\Pr_1^\bot f)|_{I_{\alpha_x}}(\theta_x)-P_t^{[0,L_1]_D}(\Pr_1^\bot f)|_{I_{\alpha_y}}(\theta_y)|=D_1+D_2.
\end{multline}
Let us first assume that  $x,y\in F_1$ belong to the same branch. Applying Lemma~\ref{L:basic_BE} with $L_0=2\pi$ and $J_0=1$ leads to
\begin{align*}
D_1&\leq\|\mcI_1 f\|_\infty\int_0^{L_0}|p_t^{F_0}(\theta,\theta_x)-p_t^{F_0}(\theta,\theta_y)|\,d\theta
\leq \min\Big\{\frac{2}{\sqrt{\pi t}},\Big(1+\frac{1}{2t}\Big)e^{-t}\Big\}\|f\|_\infty|\theta_x-\theta_y|.
\end{align*}
Notice that $\alpha_x=\alpha_y$ because $x$ and $y$ belong to the same branch. By virtue of Corollary~\ref{C:BE_interval} we thus have
\begin{align*}
D_2&= |P_t^{[0,L_1]_D}(\Pr_1^\bot f)_{\alpha_x}(\theta_x)-P_t^{[0,L_1]_D}(\Pr_1^\bot f)|_{I_{\alpha_x}}(\theta_y)|\\
&\leq \int_0^{L_1}|(\Pr_1^\bot f)_{\alpha_x}(\theta)||p_t^{[0,L_1]_D}(\theta,\theta_x)-p_t^{[0,L_1]_D}(\theta,\theta_y)|\,d\theta\\
&\leq 2\min\Big\{\frac{2}{\sqrt{\pi t}},\Big(J_1+\frac{1}{2J_1t}\Big)e^{-J_1^2t}\Big\} \|\Pr_1^\bot f\|_\infty|\theta_x-\theta_y|\\
&\leq 2\min\Big\{\frac{2}{\sqrt{\pi t}},\Big(J_1+\frac{1}{2J_1t}\Big)e^{-J_1^2t}\Big\} \| f\|_\infty d_1(x,y).
\end{align*}
Putting both estimates together we get
\begin{equation*}\label{E:BE_F1_same_branch}
|P_t^{F_1}f(x)-P_t^{F_1}f(y)|\leq C_1(t)|\theta_x-\theta_y|\| f\|_\infty =C_1(t)d_1(x,y)\| f\|_\infty 
\end{equation*}
with $C_1(t)$ as in~\eqref{E:BE_F1}, where the last equality holds because $x$ and $y$ belong to the same branch and hence $|\theta_x-\theta_y|=d_0(\phi_1(x),\phi_1(y))=d_1(x,y)$.
Suppose now that $x$ and $y$ belong to different branches. By construction, there exists a sequence of nodes $x_1,\ldots, z_N\in B_1$ with $1\leq N\leq J_1$ that joins them and such that $d_1(x,y)=d_1(x,z_1)+\sum_{\ell=1}^{N-1}d_1(x_\ell,x_{\ell+1})+d_1(z_N,y)$. 
Each pair of points in the summands belong to the same branch, hence applying the triangle inequality and estimating each term as in the previous case 
the assertion follows.
\end{proof}
\subsection{Generic approximation level}
Also here the recursive nature of the construction of $F_i$ is reflected in the proof of the weak Bakry-\'Emery condition for an arbitrary level. 

\begin{proposition}\label{T:BE_Fi}
Let $i\geq 1$. For any $t>0$, $f\in L^\infty(F_i)$ and $x,y\in F_i$,
\begin{equation}\label{E:BE_Fi}
|P_t^{F_i}f(x)-P_t^{F_i}f(y)|\leq 
2\sum_{\ell=0}^i\min\Big\{\frac{2}{\sqrt{\pi t}},\Big(J_\ell+\frac{1}{2J_\ell t}\Big)e^{-J_\ell^2t}\Big\}
d_i(x,y) \|f\|_\infty,
\end{equation}
\end{proposition}
\begin{proof}
Let $i\geq 2$. With the notation from~\eqref{E:def_fibers}, applying~\cite[Proposition 3]{AR18}, see also Remark~\ref{R:Pr_vs_fiber}, and the triangle inequality yields
\begin{multline}\label{E:BE_Fi_help01}
|P_t^{F_i}f(x)-P_t^{F_i}f(y)|\leq|P_t^{F_{i-1}}(\mcI_i f)(\phi_i(x))-P_t^{F_{i-1}}(\mcI_i f)(\phi_i(y))|\\
+|P_t^{[0,L_i]_D}(\Pr_i^\bot f)|_{I_{\alpha_x}}(\theta_x)-P_t^{[0,L_i]_D}(\Pr_i^\bot f)_{I_{\alpha_y}}(\theta_y)|=D_{i,1}+D_{i,2}.
\end{multline}
To estimate these terms, let us assume first that $x,y\in F_i$ belong to the same branch. 
By hypothesis of induction, there is a constant $C_{i-1}(t)>0$ such that
\begin{equation}\label{E:BE_Fi_help03}
D_{i,1}\leq C_{i-1}(t)\|f\|_\infty \,d_{i-1}(\phi_i(x),\phi_i(y))= C_{i-1}(t)\,d_i(x,y) \|f\|_\infty , 
\end{equation}
where the last equality is due to the fact that 
\begin{equation}\label{E:BE_Fi_help02}
d_i(x,y)=d_{i-1}(\phi_i(x),\phi_i(y))=|\theta_x-\theta_y|.
\end{equation} 
when $x$ and $y$ are in the same branch.

On the other hand, $\alpha_x=\alpha_y$ because $x$ and $y$ belong to the same branch, hence Corollary~\ref{C:BE_interval} and~\eqref{E:BE_Fi_help02} yield 
\begin{align*}
D_{i,2}&= |P_t^{[0,L_i]_D}(\Pr_i^\bot f)|_{I_{\alpha_x}}(\theta_x)-P_t^{[0,L_i]_D}(\Pr_i^\bot f)|_{I_{\alpha_x}}(\theta_y)|\\
&\leq \int_0^{L_1}|(\Pr_i^\bot f)_{\alpha_x}(\theta)||p_t^{[0,L_i]_D}(\theta,\theta_x)-p_t^{[0,L_i]_D}(\theta,\theta_y)|\,d\theta\\
&\leq 2\min\Big\{\frac{2}{\sqrt{\pi t}},\Big(J_i+\frac{1}{2J_i t}\Big)e^{-J_i^2t}\Big\}\|f\|_\infty|\theta_x-\theta_y|\\
&= 2\min\Big\{\frac{2}{\sqrt{\pi t}},\Big(J_i+\frac{1}{2J_i t}\Big)e^{-J_i^2t}\Big\}\|f\|_\infty d_i(x,y).
\end{align*}
Putting both estimates together we obtain
\begin{equation}\label{E:BE_Fi_same_branch}
|P_t^{F_i}f(x)-P_t^{F_i}f(y)|\leq\Big(C_{i-1}(t)+2\min\Big\{\frac{2}{\sqrt{\pi t}},\Big(J_i+\frac{1}{2J_i t}\Big)e^{-J_i^2t}\Big\}\Big)\|f\|_\infty d_i(x,y)
\end{equation}
for $x$ and $y$ in the same branch.
Suppose now that $x,y\in F_i$ belong to different branches. By construction, see Lemma~\ref{L:chain_cond}, we find a sequence of identification points $z_1,\ldots,z_{N_{xy}}\in B_i$ connecting the two branches so that $d_i(x,y)=d_i(x,z_1)+d_i(z_1,z_2)+\ldots+d_i(z_{N_{xy}},y)$. Moreover, the points in each pair can be regarded as belonging to the same branch. Thus, applying the triangle inequality and the previous computations to each of the terms yields again~\eqref{E:BE_Fi_same_branch}.
Finally, solving the recursive inequality $C_i(t)\leq C_{i-1}(t)+2\min\Big\{\frac{2}{\sqrt{\pi t}},\Big(J_i+\frac{1}{2J_i t}\Big)e^{-J_i^2t}\Big\}$ with $C_1(t)$ as in~\eqref{E:BE_C1} gives~\eqref{E:BE_Fi}.
\end{proof}

\subsection{Continuity estimates in the limit. Proof of Theorem~\ref{T:BE_Finfty}}
We are now ready to apply Proposition~\ref{T:BE_Fi} to obtain the estimate~\eqref{E:BE_Finfty}. Once more we see the important role that the intertwining property~\eqref{E:Intertwin} plays in order to ``pass to the limit''. Notice that, by virtue of Proposition~\ref{T:DF_props}, it suffices to prove the statement for functions $f\in\mcC_0=\bigcup_{i\geq 0}\Phi_i^*C(F_i)$.

\medskip

Let $x,y\in F_\infty$ and $f\in \mcC_0$, for instance $f\in \Phi_i^*C(F_i)$ for some $i\geq 0$. This means that there exists $h\in C(F_i)$ so that $f=h{\circ}\Phi_i$. 
By virtue of Lemma~\ref{L:Intertwin} and Theorem~\ref{T:BE_Fi} we have
\begin{align*}
|P^{F_\infty}_tf(x)-P^{F_\infty}_tf(y)|&=
|P^{F_\infty}_t\Phi_i^*h(x)-P^{F_\infty}_t\Phi_i^*h(y)|=|\Phi_i^*P_t^{F_i}h(x)-\Phi_i^*P_t^{F_i}h(y)|\nonumber\\
&=|P_t^{F_i}h(\Phi_i(x))-P_t^{F_i}h(\Phi_i(y))|\leq C_i(t)\,d_i(\Phi_i(x),\Phi_i(y))\, \|h\|_\infty .
\end{align*}
Since $d_\infty(x,y)=\lim_{i\to\infty}d_i(\Phi_i(x),\Phi_i(y))$, c.f.~\eqref{E:def_dist}, letting $i\to\infty$ yields
\begin{equation*}
|P^{F_\infty}_tf(x)-P^{F_\infty}_t f(y)|\leq C(t) \,d_\infty(x,y)\,\|f\|_\infty,
\end{equation*}
with
\begin{equation}\label{E:BE_infty_01}
C(t)\leq 2\sum_{\ell=0}^\infty\min\Big\{\frac{2}{\sqrt{\pi t}},\Big(J_\ell+\frac{1}{2J_\ell t}\Big)e^{-J_\ell^2t}\Big\}.
\end{equation}
To estimate the series on the right hand side, we notice that $J_{\ell_t^*}^{-2}\leq t <J_{\ell_t^*-1}^{-2}$ 
and split the series into the three terms
\begin{equation}\label{E:BE_infty_03}
\sum_{\ell=0}^{\ell_t^*-1}\frac{2}{\sqrt{\pi t}}+\frac{1}{\sqrt{t}}\sum_{\ell=\ell_t^*}^\infty J_\ell\sqrt{t} e^{-J_\ell^2 t}+\frac{1}{\sqrt{t}}\sum_{\ell=\ell_t^*}^\infty \frac{1}{J_\ell\sqrt{t}} e^{-J_\ell^2 t}=:\frac{2\ell_t^*}{\sqrt{\pi t}}+\frac{1}{\sqrt{t}}S_1+\frac{1}{\sqrt{t}}S_2.
\end{equation}
For the first series, analogous arguments as~\eqref{E:LipFinfty_01} give
\begin{equation*}
S_1=J_{\ell_{t^*-1}}\sqrt{t}e^{-J_{\ell_t^*}^2t}\sum_{\ell=\ell_t^*}^\infty J_{\ell_t^*,\ell}^2e^{-J_\ell^2 t}\leq \sum_{\ell=\ell_t^*}^\infty J_{\ell_t^*,\ell}^2e^{-J_{\ell_t^*+1,\ell}^2}\leq \sum_{\ell=\ell_t^*}^\infty N_{\ell_t^*}J_{\ell_t^*,\ell}^2e^{-J_{\ell_t^*+1,\ell}^2}
\end{equation*}
which is uniformly bounded independent of $t$ by Assumption~\ref{A:NJ_ass2}. For the second series,
\begin{align*}
S_2&=\sum_{\ell=\ell_t^*}^\infty\frac{1}{J_\ell\sqrt{t}}e^{-J_\ell^2t}\leq \sum_{\ell=\ell_t^*}^\infty e^{-J_{\ell_t^*+1,\ell}^2}
\end{align*}
which is finite by definition of $J_\ell$ and in particular independent of $t$. The claim now follows from~\eqref{E:BE_infty_01} and~\eqref{E:BE_infty_03}.
\subsection{Regular case}
As far as computations allow for regular diamond fractals, the estimates obtained in Theorem~\ref{T:BE_Fi} and Theorem~\ref{T:BE_Finfty} provide local continuity estimates with a logarithmic correction. 

\begin{theorem}\label{T:BE_reg_log}
For a regular diamond fractal  $F_\infty$ with parameters $n,j\geq 2$, there exists $C_{j}>0$ such that 
\begin{equation}\label{E:BE_reg_log}
|P^{F_\infty}_tf(x)-P^{F_\infty}_t f(y)|\leq \frac{C_{j}}{\sqrt{t}} (1+|\log t|)d_\infty(x,y)\|f\|_\infty 
\end{equation}
for any $f\in L^\infty(F_\infty)$, $x,y\in F_\infty$ and $0<t<1$.
\end{theorem}

\begin{proof}
To simplify constants which do not depend on $J_\ell$, $N_\ell$ or $t$, we will estimate the quantity $C_i(t)$ appearing in~\eqref{E:BE_Fi} by
\begin{equation}\label{E:BE_reg_C}
C_i(t)\leq 2\sum_{\ell=0}^i\min\Big\{\frac{1}{\sqrt{t}},\Big(J_\ell+\frac{1}{J_\ell t}\Big)e^{-J_\ell^2 t}\Big\} 
\end{equation}
for any $i=1,\ldots\infty$. Since in the regular case $J_\ell=j^\ell$, we have $j^{-2\ell_t^*}\leq t<j^{-2(\ell_t^*-1)}$, that is
\begin{align*}
\ell_t^*-1\leq\bigg|\frac{\log\sqrt{t}}{\log j}\bigg|<\ell_t^*.
\end{align*}
Estimating the terms in~\eqref{E:BE_infty_03} gives
\begin{equation*}
S_1\leq \sum_{\ell=\ell_t^*}j^{\ell-\ell_t^*}e^{-j^{2(\ell-\ell_t^*-1)}}
\leq 1+j+\int_0^\infty j^\xi e^{-j^{2\xi}}=1+j+\frac{\Gamma(1/2)}{2\log j}
\end{equation*}
and
\begin{equation*}
S_2\leq \sum_{\ell=\ell_t^*}^\infty e^{-j^{2(\ell-\ell_t^*-1)}}=e^{-\frac{1}{j^2}}+1+e^{-j^2}+\sum_{k=2}^\infty e^{-\frac{1}{j^{2}}}\leq 3+\int_1^\infty e^{-j^{2\xi}}d\xi
\leq 3+\frac{\sqrt{\pi}}{j\log j}
\end{equation*}
and we conclude from~\eqref{E:BE_infty_03} the bound
\begin{equation*}
C(t)\leq \frac{1}{\sqrt{t}}\Big(\frac{2}{\sqrt{\pi}}\ell_*+j+4+\frac{\Gamma(1/2)}{2 \log j}+\frac{\sqrt{\pi}}{j\log j}\Big)\leq \frac{C_j}{\sqrt{t}}(1+|\log t|)
\end{equation*}
that is in particular independent of $n$.
\end{proof}
\section{Applications in functional inequalities. Overview}\label{S:ineqs}
The estimates obtained in previous sections allow to analyze many other functional inequalities to further investigate the properties of the diffusion process on a generalized diamond fractal. In this section we formulate some of these and outline the main ideas to prove them.

\subsection{Ultracontractivity}
Among the different (equivalent) formulations of this property that can be found in the literature~\cite{CKS87,Dav90,Kig09,BGL14}, we consider here that of Davies~\cite[Chapter 2]{Dav90}: the semigroup $\{P^{F_\infty}_t\}_{t\geq 0}$ is contractive if it is a bounded operator from $L^2(F_\infty,\mu_\infty)$ to $L^\infty(F_\infty)$ for all $t>0$. A direct application of the estimate from Theorem~\ref{T:supHK}, and in particular that for short times in Corollary~\ref{C:supHK} leads to the desired statement.

\begin{theorem}\label{T:Ultracontractivity}
There exists $C_{\mathcal{N},\mathcal{J}}>0$ such that 
\begin{equation}\label{E:Ultracontractivity}
\|P^{F_\infty}_t\|_{2\to\infty}\leq C_{\mathcal{N},\mathcal{J}}\, t^{-\frac{1}{2}(1+d(\ell_t^*))}
\end{equation}
for any $0<t<1$.
\end{theorem}
Similarly, Corollary~\ref{C:HKsup_regular} can be applied to deduce the result in the regular case, that reads
\begin{equation*}
\|P^{F_\infty}_t\|_{2\to\infty}\leq C_{j,n}\, t^{-\frac{1}{2}(1+\frac{\log{n}}{\log j})}
\end{equation*}
with an explicit constant. Again, the spectral dimension $d_S=1+\frac{\log{n}}{\log j}$ appears in the exponent and we recover~\cite[Proposition 4.9]{HK10} without using Poincar\'e inequality.

\subsection{Poincar\'e inequality}
The ultracontractivity proved in Theorem~\ref{T:Ultracontractivity} can be applied to adapt the argument from~\cite[Proposition 4.8]{HK10} and prove a global Poincar\'e inequality in the present general (non self-similar) framework. Further inequalities of this type that require a notion of gradient, as for instance the weak (1-1) Poincar\'e inequality studied in~\cite{Laa00}, are left to be the subject of future investigations.
\begin{theorem}\label{T:uniform_PI}
A diamond fractal $F_\infty$ with parameters $\mcJ$ and $\mcN$ satisfies the uniform global
Poincar\'e inequality
\begin{equation}\label{E:uniform_PI}
\int_{F_\infty}|f-\overline{f}\,|^2\,d\mu_\infty\leq \mcE^{F_\infty}(f,f)
\end{equation}
for any $f\in\mcF^{F_\infty}$, where $\overline{f}=\frac{1}{2\pi}\int_{F_\infty}f\,d\mu_\infty$.
\end{theorem}

Since the space $(F_\infty,d_\infty)$ is compact and has finite measure, ultracontractivity implies compactness of the semigroup $P^{F_\infty}_t$ on $L^p(F_\infty,\mu_\infty)$ for any $1\leq p\leq \infty$ and $t>0$ (see e.g.~\cite[Theorem 2.1.5]{Dav90}). This can be used to deduce the existence of spectral gap~\cite[Theorem A.6.4]{BGL14} and follow~\cite[Proposition 4.8]{HK10} to obtain~\eqref{E:uniform_PI}. The lowest non-zero eigenvalue corresponds to the eigenvalue of the infinitesimal operator $L_{F_0}$, that is 
the Laplacian on the circle $F_0$; see e.g.~\cite[Proposition 2.5]{KS03} and also~\cite[Theorem 5.3]{ST12}. 

\subsection{Logarithmic Sobolev inequality}
This inequality provides 
information about the (exponential) convergence to the equilibrium of the diffusion process in terms of the \textit{entropy}, given by the expression on the left hand side of~\eqref{E:logSob}; its relation to the \textit{hypercontractivity} of the heat semigroup was showed in~\cite{Gro78}. 
\begin{theorem}\label{T:logSob}
For any non-negative function $f\in \mcF^{F_\infty}\cap L^1(F_\infty,\mu_\infty)\cap L^\infty(F_\infty)$ it holds that $f^2\log f\in L^1(F_\infty,\mu_\infty)$ and there exists $M_{\mcN\!,\mcJ}>0$ such that
\begin{equation}\label{E:logSob}
\int_{F_\infty}f^2\log f^2\,d\mu_\infty-\int_{F_\infty}f^2\,d\mu_\infty\,\log\Big(\int_{F_\infty}f\,d\mu_\infty\Big)\leq M_{\mcN\!,\mcJ}\,\mcE^{F_\infty}(f,f).
\end{equation}
\end{theorem}
One can use the estimate from Theorem~\ref{T:supHK} and classical arguments from~\cite[Theorem 2.2.3]{Dav90} to prove a defective Sobolev inequality, see e.g.~\cite[Proposition 5.1.3]{BGL14}, which by virtue of 
Theorem~\ref{T:uniform_PI} implies the logarithmic Sobolev inequality.

\section{Appendix. Functional framework}\label{S:Apx_AR18}
For the sake of completeness, this section briefly summarizes some facts obtained in~\cite{AR18} that are mentioned in some of the proofs, especially in that of Theorem~\ref{T:BE_Finfty}.

\medskip

For each $i\geq 0$, let $L^2(F_i,\mu_i)$ denote the space of square integrable functions on $F_i$. For any $i\geq 1$, we decompose this space into
\begin{equation*}\label{eq:FS.FA01}
L^2(F_i,\mu_i)=L^2_{\sym}(F_i,\mu_i)\oplus L^2_{\asym}(F_i,\mu_i),
\end{equation*}
where $L^2_{\sym}(F_i,\mu_i)$ denotes the invariant subspace of $L^2(F_i,\mu_i)$ under the action of the symmetric group $S(n_i)^{2j_i}$. 

\begin{definition}\label{def:FS.FA01}
Let $i\geq 1$. Define the projection operator $\Pr_i\colon C(F_i)\to L^2_{\sym}(F_i,\mu_i)\cap C(F_i)$  by
\begin{equation}\label{eq:FS.FA02}
\Pr_i f(x)=\begin{cases}
\frac{1}{n_i}\sum\limits_{w=1}^{n_i}f(\phi_i(x)w)&\text{if }x\in F_i\setminus B_i,\\
f(x)&\text{if }x\in B_i.
\end{cases}
\end{equation}
The orthogonal complement operator, $\Pr_i^\bot\colon C(F_i)\to L^2_{\asym}(F_i,\mu_i)\cap C(F_i)$, is defined as 
\begin{equation*}\label{eq:FS.FA03}
\Pr_i^\bot f(x)=f(x)-\Pr_i f(x).
\end{equation*}
Analogous formal definitions of these operators applies to bounded Borel functions.
\end{definition}
\begin{remark}\label{R:Pr_vs_fiber}
The projection $\Pr_i$ is related to the so-called integration over fibers in~\cite{CK15},  $\mathcal{I}_{\mcD_i}\colon C(F_i)\to C(F_{i-1})$ which in this case has the expression
\begin{equation}\label{E:def_fibers}
\mathcal{I}_if(x):=\mathcal{I}_{\mathcal{D}_i} f(x)=\frac{1}{n_i}\sum_{w=1}^{n_i}f(xw).
\end{equation}
Thus, for any $f\in C(F_i)$, 
\begin{equation*}
\Pr_i f(x)=\phi_i^*\mathcal{I}_i f(x).
\end{equation*}
\end{remark}
With the latter notation, the semigroups $\{P_t^{F_i}\}_{t\geq 0}$ admit the following decomposition, that is the $L^2$-version of~\cite[Lemma 2, Section 5]{AR18}.
\begin{lemma}\label{L:HSFA01}
For any $i\geq 1$, $t\geq 0$, $f\in\mcB_b(F_i)$ and any fixed $x\in F_i$ it holds that
\begin{equation*}\label{eq:HSFA04}
P^{F_i}_tf(x)=P_t^{F_{i-1}}(\mcI_i f)(\phi_i(x))+P_t^{[0,L_i]_D}(\Pr_i^\bot\! f)|_{I_{\alpha_x}}(\phi_{i0}(x)),
\end{equation*}
where $I_{\alpha_x}$ denotes the branch in $F_i$ where $x$ belongs to.
\end{lemma}
\section{Useful equalities and inequalities}\label{S:Apx_HKcircle}
We record the following identities relating the heat kernel on an interval and on the circle. Explicit computations can be fairly reproduced with a mathematical computing software.

\begin{lemma}\label{L:HK_formulations}
The heat kernel on the unit circle admits the representations
\begin{equation}\label{E:PoissonFormula}
p^{F_0}_t(\theta,\tilde{\theta})=\frac{1}{\sqrt{4\pi t}}\sum_{k\in\mbbZ}e^{-\frac{(\theta-\tilde{\theta}-2\pi k)^2}{4t}}=\frac{1}{2\pi}+\frac{1}{\pi}\sum_{k\geq 1}e^{-k^2t}\cos(k(\tilde{\theta}-\theta)).
\end{equation}
For any $L>0$, the heat kernel on the interval $[0,L]$ with Dirichlet boundary conditions admits the representation
\begin{equation}\label{E:1d_Dir_HK}
p_t^{[0,L]_D}(\theta,\tilde{\theta})=\frac{2}{L}\sum_{k=1}^\infty e^{-\frac{k^2\pi^2 t}{L^2}}\sin\Big(\frac{k\pi\theta}{L}\Big)\sin \Big(\frac{k\pi\tilde{\theta}}{L}\Big).
\end{equation}
Both heat kernels are related through the identity
\begin{equation}\label{E:Dir_HKvsF0}
p_t^{[0,L]_D}(\theta,\tilde{\theta})=
\frac{\pi}{L}\Big(p^{F_0}_{\pi^2t/L^2}\Big(\frac{\pi \theta}{L},\frac{\pi \tilde{\theta}}{L}\Big) - p^{F_0}_{\pi^2t/L^2}\Big(\frac{\pi \theta}{L},{-}\frac{\pi \tilde{\theta}}{L}\Big)\Big).
\end{equation}
\end{lemma}

\begin{lemma}\label{L:series_bound}
For any $a>0$, 
\begin{equation*}\label{E:series_bound}
\sum_{k=1}^\infty e^{-ak^2}\leq\min\Big\{\frac{\sqrt{\pi}}{2 \sqrt{a}},\frac{1}{a}e^{-a}\Big\}.
\end{equation*}
\end{lemma}
\begin{proof}
The series can be estimated in two different ways. On the one hand, 
\begin{equation*}
\sum_{k=1}^\infty e^{-ak^2}\leq \int_0^\infty e^{-a\xi^2}d\xi=\frac{\sqrt{\pi}}{2 \sqrt{a}}.
\end{equation*}
On the other hand, since $ak^2\geq 2ak\geq a+ak$ for any $k\geq 1$, 
\begin{equation*}
\sum_{k=1}^\infty e^{-ak^2}\leq e^{-a}\sum_{k\geq 1}e^{-ak}\leq e^{-a}\int_0^\infty e^{-a\xi}d\xi= e^{-a}\frac{1}{a}.
\end{equation*}
\end{proof}

\begin{lemma}\label{L:unif_F0}
For any $\theta,\tilde{\theta}\in[0,2\pi)$ and $t>0$,
\begin{equation*}
|p^{F_0}_t(\theta,\tilde{\theta})|\leq \frac{1}{2\pi}+\frac{1}{\sqrt{4\pi t}}.
\end{equation*}
In particular, if $t\in (0,1)$,
\begin{equation*}
|p^{F_0}_t(\theta,\tilde{\theta})|\leq \frac{1}{\sqrt{\pi t}}.
\end{equation*}
\end{lemma}
\begin{proof}
In view of the second expression in~\eqref{E:PoissonFormula},
\begin{align*}
|p^{F_0}_t(\theta,\tilde{\theta})|\leq \frac{1}{2\pi}+\frac{1}{\pi}\sum_{k\geq 1}e^{-k^2t}\leq \frac{1}{2\pi}+\frac{1}{\pi}\int_0^\infty e^{-\xi^2 t}d\xi=\frac{1}{2\pi}+\frac{1}{\sqrt{4\pi t}}.
\end{align*}
\end{proof}
\begin{lemma}\label{L:integral_bound}
For any $a>0$,
\begin{equation*}
\int_a^\infty\frac{1}{\xi^2}e^{-\xi^2t}d\xi=\frac{1}{a}e^{-a^2} - \sqrt{\pi}\operatorname{Erfc}(a)
\end{equation*}
Moreover,
\begin{equation*}
\int_a^\infty\frac{1}{\xi^2}e^{-\xi^2t}d\xi=\frac{1}{a}-\sqrt{\pi t} + a\cdot t - \frac{a^3 t^2}{6} + \frac{a^5 t^3}{30} - \frac{a^7 t^4}{168} + O(t^5)\qquad\text{as }t\to 0
\end{equation*}
\end{lemma}
\section*{Acknowledgments} The author is greatly thankful to F.\ Baudoin and A.\ Teplyaev for very valuable discussions and comments.

\bibliographystyle{amsplain}
\bibliography{DF2_Refs}
\end{document}